%% file: NAB_DEF.tex
\documentclass{ps}
\usepackage{amsmath,amsfonts,amssymb}
\usepackage{xcolor}
\usepackage[extra,safe]{tipa}
\usepackage[applemac]{inputenc}

\def \T{\mathbb{T}}

\def\esssup{{\rm esssup}}

\def\leftB{[\![}
\def\rightB{]\!]}
\include{macro}

\include{fcts-vincent}
\begin{document}
\title{On some Non Asymptotic Bounds for the Euler Scheme}
\author{V. Lemaire}\address{LPMA, Universit\'e Pierre et Marie Curie, 175 Rue du Chevaleret 75013 Paris, vincent.lemaire@upmc.fr}
\author{S. Menozzi}\address{LPMA, Universit\'e Denis Diderot, 175 Rue du Chevaleret 75013 Paris, menozzi@math.jussieu.fr}

\date{\today}
\begin{abstract} We obtain non asymptotic bounds for the Monte Carlo algorithm associated to the Euler discretization of some diffusion processes.
 The key tool is the Gaussian concentration satisfied by the density of the discretization scheme. This Gaussian concentration is derived from a Gaussian upper bound of the density of the scheme  and a modification of the so-called ``Herbst argument'' used to prove Logarithmic Sobolev inequalities.  We eventually establish a Gaussian lower bound for the density of the scheme that emphasizes the concentration is sharp. 
\end{abstract}
\maketitle

\section{Introduction}
\label{INTRO}
\subsection{Statement of the problem}
Let the $\R^d $-valued process $(X_t)_{t\ge 0}$ satisfy the dynamics
\begin{equation}
\label{SDE}
X_t=x+\bint{0}^{t} b(s,X_s) \d s+\bint{0}^t B\sigma(s,X_s) \d W_s,
\end{equation}
where $(W_t)_{t\ge 0} $ is a $d'$-dimensional ($d'\le d$) standard Brownian motion defined on a filtered probability space $(\Omega,\F,$ $(\F_t)_{t\ge 0},\P) $ satisfying the usual assumptions. The matrix $B=\left(\begin{array}{c}\mathbf{I_{d'\times d'}}\\ \mathbf{0_{(d-d')\times d'}}\end{array} \right) $ is the embedding matrix from $\R^{d'} $ into $\R^d $. The coefficients $b:\R^+\times\R^d\rightarrow \R^d,\sigma:\R^+\times \R^d\rightarrow \R^{d'}\otimes  \R^{d'}$ are assumed to be Lipschitz continuous in space, $1/2 $-H\"older continuous in time so that there exists a unique strong solution to \eqref{SDE}. 

Let us fix $T>0$ and introduce for $(t,x)\in [0,T]\times \R^d $, $Q(t,x):=\E[f(T,X_T^{t,x})]$, where $f$ is a measurable function, bounded in time and with polynomial growth in space.  The numerical approximation of $Q(t,x)$ appears in many applicative fields. In mathematical finance, $Q(t,x)$ can be related to the price of an option when the underlying asset follows the dynamics \eqref{SDE}. In this framework we consider two important cases:
\begin{trivlist}
\item[(a)] If $d=d'$, $Q(t,x)$ corresponds to the price at time $t$ when $X_t=x$ of the vanilla option with maturity $T$ and pay-off $f$.
\item[(b)] If $d'=d/2$, $b(x)=\left(\begin{array}{c}b_1(x)\\b_2(x)\end{array}\right)$ where $b_1(x)\in \R^{d'}, b_2(x)=(x_1,\cdots,x_{d'})^*$, $Q(t,x)$ corresponds to the price of an Asian option.
\end{trivlist}
It is also well known, see e.g. Friedman \cite{frie:75}, that $Q(t,x)$ is the Feynman-Kac representation of the solution of the parabolic PDE 
\begin{equation}
\label{EDP}
\left\{\begin{array}{l}
\partial_t Q(t,x)+LQ(t,x)=0,\ (t,x)\in [0,T) \times \R^d,\\
Q(T,x)=f(T,x), \ x \in \R^d, 
\end{array}
\right.
\end{equation}
where $L$ stands for the infinitesimal generator of \eqref{SDE}. 
Hence, the quantity $Q(t,x) $ can also be related to problems of heat diffusion with Cauchy boundary conditions (case (a)) or to kinetic systems (case (b)).  

The natural probabilistic approximation of $Q(t,x) $ consists in considering the Monte Carlo algorithm. This approach is particularly relevant compared to deterministic methods if the dimension $d$ is large.
To this end we introduce some discretization schemes. For case (a) we consider the Euler scheme with time step $\Delta:=T/N,\ N\in \N^*$. Set $\forall i\in \N,\ t_i=i\Delta $ and for $t\ge 0$, define $\phi(t)=t_i $ for $t_i\le t<t_{i+1} $. The Euler scheme writes
\begin{equation}
\label{EUL}
X_t^\Delta=x+\bint{0}^{t} b(\phi(s), X_{\phi(s)}^\Delta) \d s +\bint{0}^{t} \sigma(\phi(s), X_{\phi(s)}^\Delta) \d W_s.
\end{equation}
For case (b) we define 
\begin{equation}
\label{EUL_MODIF}
 X_t^\Delta=x+\bint{0}^{t} \left(\begin{array}{c} b_1(\phi(s),   X_{\phi(s)}^\Delta)\\ (  X_s^\Delta)^{1,d'}\end{array} \right) \d s +\bint{0}^{t} 
 B\sigma(\phi(s),   X_{\phi(s)}^\Delta)\d W_s,
\end{equation}
where $(  X_s^\Delta)^{1,d'}:=\bigl((  X_s^\Delta)^1,\cdots ,(  X_s^\Delta)^{d'}\bigr)^*$. Equation \eqref{EUL_MODIF} defines a completely simulatable scheme with Gaussian increments. On every time step, the last $d'$ components are the integral of a Gaussian process. 

The weak error for the above problems has been widely investigated in  the literature. Under suitable assumptions on the coefficients $b,\sigma $ and $f$ (namely smoothness)  it is shown in Talay and Tubaro \cite{tala:tuba:90} that
$E_D(\Delta):=\E_x[f(T,X_T^\Delta)]-\E_x[f(T,X_T)]=C\Delta +O(\Delta^2)$. Bally and Talay \cite{ball:tala:96:1} then extended this result to the case of bounded measurable functions $f$ in a  hypoelliptic setting for time homogeneous coefficients $b,\sigma$.
Also, still for time homogeneous coefficients, similar expansions have been derived for the difference of the densities of the process and the discretization scheme, see Konakov and Mammen \cite{kona:mamm:02} in case (a), Konakov \textit{et al.} \cite{kona:meno:molc:09} in case (b) for a uniformly elliptic diffusion coefficient $\sigma\sigma^*$, and eventually Bally and Talay \cite{ball:tala:96:2} for a hypoelliptic diffusion and a slight modification of the Euler scheme.
 The constant $C$ in the above development involves the derivatives of $Q$ and therefore depends on $f,b,\sigma,x$. 
 
 The expansion of $E_D(\Delta) $ gives a good control on the impact of the discretization procedure of the initial diffusion, and also permits to improve the convergence rate using e.g. Richardson-Romberg extrapolation (see \cite{tala:tuba:90}). 
Anyhow, to have a global sharp control of the numerical procedure it remains to consider the quantities 
\begin{equation}
\label{EST_MC}
E_{MC}(M,\Delta) = \ds \frac 1M\bsum{i=1}^{M} f(T,(X_{T}^\Delta)^i)-\esps{x}{f(T,X_T^\Delta)}. 
\end{equation}
In the previous quantities $M$ stands for the number of independent samples in the Monte Carlo algorithm and $\bigl( (X_t^\Delta)_{t\ge 0}^i \bigr )_{i\in\leftB 1,M\rightB}$ are independent sample paths. Indeed, the global error associated to the Monte Carlo algorithm writes:
$$E(M,\Delta)=E_{D}(\Delta)+E_{MC}(M,\Delta),$$
where $E_{D}(\Delta) $ is the \textit{discretization error} and $E_{MC}(M,\Delta) $ is the pure \textit{Monte Carlo} error.

The convergence of $E_{MC}(M,\Delta)$, to $0$ when $M\rightarrow \infty$ is ensured under the above assumptions on $f$ by the strong law of large numbers. A speed of convergence can also be derived from the central limit theorem, but these results are asymptotic, i.e. they hold for a sufficiently large $M$. On the other hand, a non asymptotic result is provided by the Berry-Esseen Theorem that compares the distribution function of the normalized Monte Carlo error to the distribution function of the normal law at order $O(M^{-1/2})$.

In the current work we are interested in giving, for Lipschitz continuous in space functions $f$, non asymptotic error bounds for the quantity $E_{MC}(M,\Delta)$. 
Similar issues had previously been studied by Malrieu and Talay \cite{malr:tala:06}. In that work, the authors investigated the concentration properties of the Euler scheme and obtained Logarithmic Sobolev inequalities, that imply Gaussian concentration see e.g. Ledoux \cite{ledo:99}, for multi-dimensional Euler schemes with constant diffusion coefficients. 
Their goal was in some sense different than ours since they were mainly interested in ergodic simulations. In that framework we also mention the recent work of Joulin and Ollivier for Markov chains \cite{joul:olli:09}.

Our  strategy is here different. We are interested in the approximation of $Q(t,x),\ t\le T$ where $T>0$ is fixed. It turns out that the log-Sobolev machinery is in some sense too rigid and too ergodic oriented. Also, as far as approximation schemes are concerned it seems really difficult to obtain log-Sobolev inequalities in dimension greater or equal than two without the constant diffusion assumption, see \cite{malr:tala:06}. Anyhow, under suitable assumptions on $b,\sigma$ (namely uniform ellipticity of $\sigma\sigma^* $ and mild space regularity), the discretization schemes \eqref{EUL}, \eqref{EUL_MODIF} can be shown to have  a density admitting a Gaussian upper bound. From this \textit{a priori} control we can modify Herbst's argument to obtain an expected Gaussian concentration as well as the tensorization property (see \cite{ledo:99}) that will yield for $r>0$ and a Lipschitz continuous in space $f$, $\prob[1]{|E_{MC}(M,\Delta)|\ge r+\delta} \le 2 \exp(-\frac{M}{\alpha(T)} r^2)$ for $\alpha(T) > 0$ independent of $M$ uniformly in $\Delta=T/N$. Here $\delta \ge 0$ is a bias term (independent of $M$) depending on the constants appearing in the Gaussian domination (see Theorem \ref{Aronson_Euler}) and on the Wasserstein distance between the law of the discretization scheme and the Gaussian upper bound. We also prove that a Gaussian lower bound holds true for the density of the scheme. Hence, the Gaussian concentration is sharp, i.e. for a function $f$ with suitable non vanishing behavior at infinity, the concentration is at most Gaussian, i.e. $\prob[1]{|E_{MC}(M,\Delta)| \ge r - \bar \delta} \ge 2\exp(-\frac M{\bar \alpha(T)} r^2)$, for $r$ large enough, $\bar \delta$ depending on $f$, and the Gaussian upper and lower bounds, $\bar \alpha(T) >0$ independent of $M$ uniformly in $\Delta=T/N$.

The paper is organized as follows, we first give our standing assumptions and some notations in Section \ref{ASS_AND_NOT}. We state our main results in Section \ref{RES}. Section \ref{CONC} is dedicated to concentration properties and non asymptotic Monte Carlo bounds for random variables whose law admits a density dominated by a probability density satisfying a log-Sobolev inequality.
We prove our main deviations results at the end of that section as well.  
In Section \ref{PARAM} we show how to obtain the previously mentioned Gaussian bounds in the two cases introduced above. The main tool for the upper  bound is a discrete parametrix representation of Mc Kean-Singer type for the density of the scheme, see \cite{mcke:sing:67} and Konakov and Mammen \cite{kona:mamm:00} or \cite{kona:mamm:02}. The lower bound is then derived through suitable chaining arguments adapted to our non Markovian setting.

\subsection{Assumptions and Notations}
 \label{ASS_AND_NOT}
We first specify some assumptions on the coefficients. Namely, we assume:
\begin{trivlist}
\item[\A{UE}] The diffusion coefficient is uniformly elliptic. There exists $\lambda_0>0$ s.t. for $(t,x,\xi)\in [0,T]\times \R^d\times \R^{d'}$ we have $\lambda_0^{-1}|\xi|^2\le  \langle a(t,x)\xi,\xi\rangle \le \lambda_0 |\xi|^2$ where $a(t,x):=\sigma\sigma^*(t,x)$, and $\abs{.}$ stands for the Euclidean norm.
\item[\A{SB}] The diffusion matrix $a$ is uniformly $\eta $-H\"older continuous in space, $\eta>0 $, uniformly in time, and the drift $b$ is bounded. That is there exists $L_0>0$ s.t.
$$\sup_{(t,x)\in [0,T]\times \R^d}|b(t,x)|+\sup_{t\in [0,T], (x,y)\in \R^{2d},\ x\neq y} \frac{|a(t,x)-a(t,y)|}{|x-y|^\eta}\le L_0.$$
\end{trivlist}
Throughout the paper we assume that \A{UE}, \A{SB} are in force.

In the following we will denote by $C$ a generic positive constant that can depend on $L_0,\lambda_0,\eta,d,T$. We reserve the notation $c$ for constant depending on $L_0,\lambda_0,\eta,d$ but not on $T$.
In particular the constants $c,C$ are uniform w.r.t the discretization parameter $\Delta=T/N$ and eventually the value of both $c,C$ may change from line to line. 

To establish concentration properties, we will work with the class of Lipschitz continuous functions $F:\R^d\rightarrow \R$ satisfying $\abs{\nabla F}_\infty = \esssup_x \abs{\nabla F(x)} \le 1$ where $\abs{\nabla F}$ denotes the Euclidean norm of the gradient $\nabla F$, defined almost everywhere, of $F$.

Denote now by $S^{d-1} $ the unit sphere of $\R^d$. For $z\in \R^d\backslash \{ 0\},\ \pi_{S^{d-1}}(z)$ stands for the uniquely defined projection on $S^{d-1}$.
For given $\rho_0 > 0$, $\beta > 0$, we introduce the following growth assumption in space for $F$ in the above class of functions: 
\begin{trivlist}
\item[$\A{G_{\rho_0, \beta} }$] There exists $A \subset S^{d-1}$ such that 
$$ \forall y\in \R^d\backslash B(\rho_0), \pi_{S^{d-1}}(y) \in A,\ y_0:=\rho_0 \pi_{S^{d-1}}(y),  F(y)-F(y_0) \ge \beta |y-y_0|, $$
\end{trivlist}
with $A$ of non empty interior and $\abs{A} \ge \varepsilon > 0$ for $d \ge 2$ ($\abs{.}$ standing here for the Lebesgue measure of $S^{d-1}$), and $A \subset \ac{-1,1}$ for $d=1$. In the above equation $B(\rho_0)$ stands for the Euclidean ball of $\R^d$ of radius $\rho_0 $, and $\psca{.,.}$ denotes the scalar product in $\R^d$.

\begin{REM}
The above assumption simply means that for $|y|\ge \rho_0$ the graph of $F$ stays above a given hyperplane. In particular, for all $z\in A, F(r z )\underset{r\rightarrow +\infty}{\rightarrow}+\infty$.
\end{REM}

The bounds of the quantities $E_{MC}(M,\Delta)$ will be established for real valued functions $f$ that are uniformly Lipschitz continuous in space and measurable bounded in time, such that for a fixed $T$, $F(.) := f(T,.)$ will be Lispchitz continuous satisfying $\abs{\nabla F}_\infty \le 1$. 
Moreover, for the lower bounds, 
we will suppose that the above $F$ satisfies $\A{G_{\rho_0, \beta} }$.

\mysection{Results}
\label{RES}
Let us first justify that under the assumptions \A{UE}, \A{SB}, the discretization schemes admit a density. For all $x\in \R^d$, $0\le j<j'\le N $, $A\in {\cal B}(\R^d) $ (where ${\cal B}(\R^d)$ stands for the Borel $\sigma$-field of $\R^d$) we get 
\begin{multline}
\label{EXPR_DENS}
\probc[1]{X_{t_{j'}}^\Delta \in A}{X_{t_j}^\Delta=x}
= \bint{(\R^d)^{j'-j-1}\times A}^{} p^{\Delta}(t_j,t_{j+1}, x, x_{j+1})p^{\Delta}(t_{j+1},t_{j+2},x_{j+1},x_{j+2})\cdots p^{\Delta}(t_{j'-1},t_{j'},x_{j'-1},x_{j'})\\
\times \d x_{j+1}\d x_{j+2}\cdots \d x_{j'},
\end{multline}
where the notation $p^{\Delta}(t_i,t_{i+1},x_i,x_{i+1}), \ i\in \leftB 0,N-1\rightB $ stands in case (a) for the density at point $x_{i+1}$ of a Gaussian random variable with mean $x_i+b(t_i,x_i)\Delta$ and non degenerated covariance matrix $a(t_i,x_i) \Delta$, whereas in case (b) it stands for the density of a Gaussian random variable with mean $\left( \begin{array}{c}
x_i^{1,d'} + b_1(t_i,x_i)\Delta ,\\
 x_i^{d'+1,d} + x_i^{1,d'}\Delta+b_1(t_i,x_i)\Delta^2/2 
\end{array}\right) $ and non degenerated as well covariance matrix $\left(\begin{array}{ll}a(t_i,x_i) \Delta & a(t_i,x_i) \Delta^2/2\\a(t_i,x_i) \Delta^2/2 & a(t_i,x_i) \Delta^3/3 \end{array}\right)$, where $\forall y \in \R^d$, $y^{1,d'} = (y^1,\dots,y^{d'})^*$ and $y^{d'+1,d} = (y^{d'+1},\dots,y^{d})^*$.  

Equation \eqref{EXPR_DENS} therefore guarantees the existence of the density for the discretization schemes.
From now on, we denote by $p^{\Delta}(t_j,t_{j'},x,\cdot)$ the transition densities between times $t_j$ and $t_{j'},\ 0\le j<j'\le N$, of the discretization schemes \eqref{EUL}, \eqref{EUL_MODIF}. Let us denote by $\P_x$ (resp. $\P_{t_j,x},\ 0\le j<N $) the conditional probability given $\ac{X_0^\Delta = x}$ (resp. $\{X_{t_j}^\Delta=x\}$), so that in particular $\probs{x}{X_T^\Delta \in A} = \int_A p^\Delta(0,T, x, x') \d x'$. We have the following Gaussian estimates for the densities of the schemes.
\begin{THM}[``Aronson'' Gaussian estimates for the discrete Euler scheme]
\label{Aronson_Euler}
Assume \A{UE}, \A{SB}. There exist constants $c>0,C \ge 1$, s.t. for every $0\le j<{j'}\le N $:
\begin{equation}
\label{Aronson_E}
 C^{-1}  p_{c^{-1}}(t_{j'}-t_j,x,x')\le {p}^{\Delta}(t_j,t_{j'},x,x')\le C  p_c(t_{j'}-t_j,x,x'),
\end{equation}
where for all $0\le s<t\le T $, in case (a), $p_c(t-s,x,x'):=\left(\frac{c}{2\pi (t-s) }\right)^{d/2}\exp\bigl(-c\frac{|x'-x|^2}{2(t-s)}\bigr) $ and in case (b)
$$  p_c(t-s,x,x'):= \pa{\frac{\sqrt{3} c}{2\pi (t-s)^{2}}}^{d/2}
\exp\biggl(-c\biggl\{\frac{|(x')^{1,d'}-x^{1,d'}|^2}{4(t-s)}+3\frac{|(x')^{d'+1,d}-x^{d'+1,d}-\frac{x^{1,d'}+(x')^{1,d'}} 2(t-s)|^2}{(t-s)^3}\biggr\}\biggr).$$

Note that $p_c$ enjoys the semigroup property, i.e. $\forall 0<s<t,\ \int_{\R^d}^{}   p_c(t-s, x,u)   p_c(s,u,x')\d u=  p_c(t,x,x')$ (see Kolmogorov \cite{kolm:33} or \cite{kona:meno:molc:09} for case (b)). 
\end{THM}

\begin{REM}
The above upper bound can be found in \cite{kona:mamm:02} in the case of time homogeneous Lipschitz continuous coefficients.
Both bounds can be derived (for time dependent coefficients) from the work of Gobet and Labart \cite{gobe:laba:08} under stronger smoothness assumptions. Here, our framework is the one of the
``standard" PDE assumptions to derive Aronson's estimates for the fundamental solution of non degenerated non-divergence form second order operators, see e.g. Sheu \cite{sheu:91} or \cite{dela:meno:09}. In particular no regularity in time is needed.
\end{REM}

Our second result is the Gaussian concentration of the Monte Carlo error $E_{MC}(M, \Delta)$ defined in \eqref{EST_MC} for a fixed $M$ uniformly in $\Delta=T/N, \ N\ge 1$.

\begin{THM}[Gaussian concentration] \label{MTHM}
Assume \A{UE}, \A{SB}. For the constants $c$ and $C$ of Theorem \ref{Aronson_Euler}, we have for every $\Delta =T/N,\ N\ge 1$, and every Lipschitz continuous function in space and measurable bounded in time $f:\R^d\rightarrow \R $ satisfying $\abs{\nabla f(T,.)}_\infty \le 1$ in \eqref{EST_MC},  
\begin{equation} \label{MAJO-THM}
	\forall r > 0, \quad \forall M \ge 1, \quad \probs{x}{\abs[1]{E_{MC}(M, \Delta)} \ge r + \delta_{C, \alpha(T)}} \le 2 e^{- \frac{M}{\alpha(T)} r^2}, 
\end{equation}
with 
\begin{equation} \label{DEF_ALPHA_T}
	\frac{1}{\alpha(T)} = \begin{cases}
		\frac{c}{2T} & \text{in case (a)}, \\
		\frac{c}{2T}\pa{1 + \frac{3}{T^2}\pa[2]{1-\sqrt{1+\frac{T^2}{3} + \frac{T^4}{9}}}} & \text{in case (b)}, \\
	\end{cases}
\end{equation}
and $\delta_{C, \alpha(T)} = 2 \sqrt{\alpha(T) \log C}$.

Moreover, if $F(.):=f(T,.) \ge 0$ satisfies for a given $\rho_0>0$ and $\beta > 0$, the growth assumption $\A{G_{\rho_0, \beta}}$,
\begin{equation} \label{MINO-THM}
	\forall r > 0, \quad \forall M \ge 1, \quad \probs{x}{\abs[1]{E_{MC}(M, \Delta)} \ge r - \bar \delta_{c,C,T, f}} \ge 2 \exp\pa{- \frac M {\bar \alpha(T)} \left[\frac r\beta\vee \rho_0  \right]^2},
\end{equation}
where $\bar \delta_{c,C, T, f} = (1 + \sqrt{2}) \sqrt{\alpha(T) \log C} + \gamma_{c^{-1},T}(F) + \rho_0\beta-\underline{F},\ \gamma_{c^{-1},T}(\d x')=p_{c^{-1}}(T,x,x')\d x'$, and $ \underline{F}:=\inf_{s\in S^{d-1}}F(s\rho_0)$. The constant $\bar \alpha(T)^{-1} $ appearing in \eqref{MINO-THM} writes  in case (a)
\begin{eqnarray*} 
	 \frac{1}{\bar \alpha(T)}&=&\bar \Lambda+\chi:=	\begin{cases}
 \frac {c^{-1}}{2T}+\frac{1}{\rho^2_0}\log\pa{\frac{\pi^{d/2}   C  }{K(d,A)                }}_+  & \text{for $d $ even},\\
  \frac{c^{-1}\theta}{2T}+ \frac{1}{\rho_0^2}\log\pa{\frac{\pi^{d/2} C}{ \arccos(\theta^{-1/2}) K(d,A)}}_+& \text{for $d $ odd, $\theta \in (1,+\infty)$},
  \end{cases}
\end{eqnarray*}
where for all $d \in \N^*$, $A\subset S^{d-1}$ appearing in $\A{G_{\rho_0, \beta}}$,
\begin{equation} \label{K_d_A}
K(d,A)=\begin{cases} \frac{|A| (d/2-1)!}{2},\ d \text{ even},\\
\frac{|A|\prod_{j=1}^{\frac{d-1}2}(j-1/2)}{\pi^{1/2}},\ d \text{ odd}.
\end{cases} 
\end{equation}
In case (b), $d$ is even and 
\begin{eqnarray*}
\frac 1{\bar \alpha(T)}=\bar \Lambda+\chi:=\frac{c^{-1}}{2T}\left(1+\frac 3 {T^2}\left[1+\sqrt {1+\frac{T^2}3+\frac {T^4} 9}  \right] \right)+\frac{1}{\rho^2_0}\log\pa{\frac{ \left( \frac \pi T  \right)^{d/2}[T^2+3(1+\sqrt {1+\frac{T^2}3+\frac {T^4} 9})]^{d/2}  C  }{K(d,A)                }}_+.
\end{eqnarray*}
\end{THM}

From Theorem \ref{Aronson_Euler} and our current assumptions on $f$, we can deduce from the central limit theorem that $M^{1/2}E_{MC}(M,\Delta)\overset{({\rm law})}{\underset{M}{ \rightarrow}}{\cal N}(0,\sigma^2(f,\Delta)),\ \sigma^2(f,\Delta):=\E_x[f(X_T^\Delta)^2]-\E_x[f(X_T^\Delta)]^2$. From this asymptotic regime, we thus derive that for large $M$ the typical deviation rate $r$ (i.e. the size of the confidence interval) in \eqref{MAJO-THM} has order $c\sigma(f,\Delta)M^{-1/2}$ where for a given threshold $\alpha\in (0,1) $, $c:=c(\alpha)$ can be deduced from the inverse of the Gaussian distribution function. In other words, $r$ is typically small for large $M$. On the other hand, we have a systematic bias $\delta_{C,\alpha(T)} $, independently of $M$. In whole generality, this bias is inherent to the concentration arguments used to derive the above bounds, see Section \ref{CONC}, and cannot be avoided. Hence, those bounds turn out to be particularly relevant to derive non asymptotic confidence intervals when $r$ and $\delta_{C,\alpha(T)} $ have the same order. In particular, the parameter $M$ is not meant to go to infinity. This kind of result can be useful if for instance it is particularly heavy to simulate the underlying Euler scheme and that only a relatively small number $M$ of samples is reasonably allowed. On the other hand, the smaller $T$ is the bigger $M$ can be.  Precisely, one can prove that the constant $C$ of Theorem \ref{Aronson_Euler} is bounded by $\bar c \exp(L_0 T)$ (see Section \ref{PARAM}). Hence from \eqref{DEF_ALPHA_T}, we have $\delta_{C,\alpha(T)} = O(T)$ for $T$ small. 


\begin{REM}
For the lower bound, the ``expected" value for $\bar \alpha(T)^{-1}$ would be  $\bar \lambda$ corresponding to the largest eigenvalue of one half the inverse of the covariance matrix of the random variable with density $p_{c^{-1}}(T,x,.)$ appearing in the lower bound of Theorem \ref{Aronson_Euler}. There are two corrections with respect to this intuitive approach. First, there is in case (a) an additional multiplicative term $\theta>1$ (that can be optimized) when $d$ is odd. This correction is unavoidable for $d=1$, anyhow for odd $d>1$, it can be avoided up to an additional additive factor like the above $\chi$ (see the proof of Proposition \ref{SHARP_MINO} for details). We kept this presentation to be homogeneous for all odd dimensions.
 
 Also, an additive correction (or penalty) factor $\chi$ appears. It is mainly due to our growth assumption $\A{G_{\rho_0,\beta}}$.
Observe anyhow that, for given $T>0,C\ge 1,\varepsilon>0$ s.t. $|A|\ge \varepsilon $, 
 if the dimension $d$ is large enough, by definition of $K(d,A)$, we have $\chi=0 $. Still, for $d=1$ (which can only occur in case (a)) we cannot avoid the correction factor $\chi $.  
\end{REM}
\begin{REM}
Let us also specify that in the above definition of $\chi$, $\rho_0$ is not meant to go to zero, even though some useful functions like $|.|$ satisfy $\A{G_{\rho_0, 1} }$ with any $\rho_0>0$. 
Actually, the bound is particularly relevant in `large regimes", that is when $r/\beta $ is not assumed to be small. Also, we could replace in the above definition of $\chi$, $\rho_0$ by $R>0 $ as soon as $r/\beta\ge R $. In particular, if $F$ satisfies $\A{G_{\rho_0,\beta}}$, for $R\ge \rho_0$ it also satisfies $\A{G_{\textit{R},\beta}}$. We gave the statement with $\rho_0 $ in order to be uniform w.r.t. the threshold $\rho_0 $ appearing in the growth assumption of $F$ but the correction term can be improved in function of the  deviation factor $r/\beta $.
\end{REM}

\begin{REM}
 Note that under \A{UE}, \A{SB}, in case (a), the martingale problem in the sense of Stroock and Varadhan is well posed for equation \eqref{SDE}, see Theorem 7.2.1 in \cite{stro:vara:79}. Also, from Theorem \ref{Aronson_Euler} and the estimates of Section \ref{PARAM}, one can deduce that the unique weak solution of the martingale problem has a smooth density that satisfies Aronson like bounds. Furthermore, a careful reading of \cite{kona:mamm:02} emphasizes that the discretization error analysis carried therein can be extended to our current framework, i.e. we only need boundedness of the drift and uniform spatial H\"older continuity of the (non-degenerated) diffusion coefficient to control  $E_D(\Delta) $. Hence, the above concentration result gives that in case (a), one can control the global error $E(M,\Delta):=E_D(\Delta)+E_{MC}(M,\Delta)$. The well-posedness of the martingale problem in case (b) remains to our best knowledge an open question and will concern further research. 
\end{REM}

\begin{REM}
In case (b), the concentration regime in the above bounds highly depends on $T$. Since the two components do not have the same scale we have that, in short time, the concentration regime is the one of the non degenerated component in the upper bound (resp. of the degenerated component in the lower bound). For large $T$, it is the contrary. 
\end{REM}

We now consider an important case for applications in case (b). Namely, in kinetic models (resp. in financial mathematics) it is often useful to evaluate the expectation of functions that involve the difference of the first component and its normalized average (which corresponds to a time normalization of the second component).  This allows to compare the velocity (resp. the price) at a given time $T$ and the averaged velocity (resp. averaged price) on the associated time interval.
Obviously, the normalization is made so that the two components have  time-homogeneous scales. We have the following result.
\begin{COROL} \label{corol-asiat}
In case $(b)$, if $f$ in \eqref{EST_MC} writes $f(T,x) = g(T, \T_T^{-1} x)$ where $\T_T^{-1} = \pa{\begin{array}{c c} \mathbf{I_{d' \times d'}} & \mathbf{0_{d'\times d'}}\\ \mathbf{0_{d'\times d'}} & T^{-1} \mathbf{I_{d' \times d'}} \end{array}}$ and $g$ is a Lipschitz continuous function in space and measurable bounded in time satisfying $\abs{\nabla g(T, .)}_\infty \le 1$ then we have for every $\Delta=T/N,\ N\ge 1$,    
\begin{equation*}
	\forall M \ge 1, \quad \probs{x}{\abs[1]{E_{MC}(M, \Delta)} \ge r + \delta_{C, \alpha(T)}} \le 2 e^{- \frac{M}{\alpha(T)} r^2}, 
\end{equation*}
with $\alpha(T) = (4-\sqrt{13}) \frac{c}{T}$ and $\delta_{C, \alpha(T)} = 2 \sqrt{\alpha(T) \log C}$.
\end{COROL}
A lower bound could be derived similarly to Theorem \ref{MTHM}.

The proof of Theorems \ref{Aronson_Euler} and \ref{MTHM} (as well as Corollary \ref{corol-asiat}) are respectively postponed to Sections \ref{SEC_PREUVE_AR} and \ref{PROOF_MTHM}.

\mysection{Gaussian concentration and non asymptotic Monte Carlo bounds} 
\label{CONC}

\subsection{Gaussian concentration - Upper bound}
We recall that a probability measure $\gamma$ on $\R^d $ satisfies a logarithmic Sobolev inequality with constant $\alpha > 0$ if for all $f \in H^1(\d \gamma):=\{g\in L^2(\d\gamma): \int \abs{\nabla g}^2 \d\gamma<+\infty\}$ such that $f \ge 0$, one has 
\begin{equation} \tag{$\mathcal{LSI}_\alpha$} \label{ISL}
	\ent_\gamma(f^2) \le \alpha \int \abs{\nabla f}^2 \d \gamma,
\end{equation} 
where $\ent_\gamma(\phi) = \int \phi \log(\phi) \d \gamma - \pa{\int \phi \d \gamma} \log\pa{\int \phi \d \gamma}$ denotes the entropy of the measure $\gamma$. In particular, we have the following result (see \cite{ledo:99} Section 2.2 eq. (2.17)). 
\begin{PROP} \label{THMCRITERE} 
Let $V$ be a $\mathcal{C}^2$ convex function on $\R^d$ with $\Hess V \ge \lambda  \mathbf{I_{d \times d}}$, $\lambda > 0$ and such that $e^{-V}$ is integrable with respect to the Lebesgue measure. Let $\gamma(\d x) = \frac{1}{Z} e^{-V(x)} \d x$ be a probability measure (Gibbs measure). Then $\gamma$ satisfies a logarithmic Sobolev inequality with constant $\alpha = \frac{2}{\lambda}$.
\end{PROP}

Throughout this section we consider a probability measure $\mu$ with density $m$ with respect to the Lebesgue measure $\lambda_K$ on $\R^K$ (here we have in mind $K = d$ or $K = Md$, $M$ being the number of Monte Carlo paths). We assume that $\mu$ is dominated by a probability measure $\gamma$ in the following sense 
\begin{equation} \label{DOMI} \tag{$\mathcal{H}_{\kappa, \alpha}$}
 	\gamma(\d x) = q(x) \d x \;\text{ satisfies }\; \eqref{ISL} \quad \text{and} \quad \exists \kappa \ge 1, \quad \forall x \in \R^K, \quad m(x) \le \kappa q(x).
\end{equation}

\begin{PROP} \label{PROP1}
Assume that $\mu$ and $\gamma$ satisfy \eqref{DOMI}. Then for all Lipschitz continuous function $F:\R^K\rightarrow \R$ s.t. $\abs{\nabla F}_{\infty} \le 1$, 
\begin{equation*}
	\forall r > 0, \quad \probs[2]\mu{F(Y) - \mu(F) \ge r + W_1(\mu, \gamma)} \le \kappa e^{-\frac{r^2}{\alpha}}, 
\end{equation*}
where $W_1(\mu, \gamma) = \ds \sup_{\abs{\nabla F}_{\infty} \le 1} \abs{\mu(F) - \gamma(F)}$ (Wasserstein distance $W_1 $ between $\mu $ and $\gamma $).
\end{PROP}

\begin{proof}
By the Markov inequality, one has for every $\lambda > 0$,
\begin{equation}
	\probs[2]\mu{F(Y) - \mu(F) \ge r + W_1(\mu, \gamma)} \le e^{-\lambda (\mu(F) + r + W_1(\mu, \gamma))} \esps[2]\mu{e^{\lambda F(Y)}}, \label{domi2}
\end{equation}
and by \eqref{DOMI}, $\esps[1]\mu{e^{\lambda F(Y)}} \le \kappa \esps[1]\gamma{e^{\lambda F(Y)}}$. 
Since $\gamma$ satisfies a logarithmic Sobolev inequality with constant $\alpha > 0$, the Herbst argument (see e.g. Ledoux \cite{ledo:99} section 2.3) gives 
\begin{equation*}
	\esps[2]\gamma{e^{\lambda F(Y)}} \le e^{\lambda \gamma(F) + \frac{\alpha}{4} \lambda^2},
\end{equation*}
so that $\esps[1]\mu{e^{\lambda F(Y)}} \le \kappa e^{\lambda \mu(F) + \frac{\alpha}{4} \lambda^2 + \lambda \pa{\gamma(F) - \mu(F)}}$ and  
\begin{equation} \label{control_laplace}
	\esps[2]\mu{e^{\lambda F(Y)}} \le \kappa e^{\lambda \mu(F) + \frac{\alpha}{4} \lambda^2 + \lambda W_1(\mu, \gamma)},
\end{equation}
since owing to the definition of $W_1$ one has $W_1(\mu, \gamma) \ge \gamma(F) - \mu(F)$.
Plugging the above control \eqref{control_laplace} into \eqref{domi2} yields 
\begin{equation*}
	\probs[2]\mu{F(Y) - \mu(F) \ge r + W_1(\mu, \gamma)} \le \kappa e^{-\lambda r + \frac{\alpha}{4}\lambda^2}.
\end{equation*}
An optimization on $\lambda$ gives the result.
\end{proof}

\begin{LEMME} \label{LEMW1}
Assume that $\mu$ with density $m$ and $\gamma$ with density $q$ satisfy the domination condition
\begin{equation*}
	\exists \kappa \ge 1, \quad \forall x \in \R^d, \quad m(x) \le \kappa q(x)
\end{equation*}
and that there exist $(\alpha, \beta_1, \beta_2) \in (\R_+)^3$ such that for all Lipschitz continuous function $F$ satisfying $\abs{\nabla F}_\infty \le 1$ and for all $\lambda > 0$, $\ds \esps{\gamma}{e^{\lambda F(Y)}} \le e^{\lambda \gamma(F) + \frac{\alpha}{4} \lambda^2 + \beta_1 \lambda + \beta_2}$. Then we have, $W_1(\mu, \gamma) \le \beta_1 + \sqrt{\alpha \pa{\beta_2 + \log(\kappa)}}$.
\end{LEMME}
\begin{proof}
Recall first that for a non-negative function $f$, we have the following variational formulation of the entropy:
\begin{equation} \label{var_entropy}
	\ent_\gamma(f) = \sup\ac{\esps{\gamma}{fh}; \;\esps{\gamma}{e^h} \le 1}.
\end{equation}

W.l.o.g. we consider $F$ such that $\mu(F) \ge \gamma(F)$. Let $\lambda > 0$ and $h := \lambda F - \lambda \gamma(F) - \frac{\alpha}{4} \lambda^2 - \beta_1 \lambda - \beta_2$ so that $\esps{\gamma}{e^{h}} \le 1$ and 
\begin{equation*}
	\esps{\mu}{h} = \esps{\gamma}{\frac{m}{q} h} = \lambda \pa{\mu(F) - \gamma(F)} - \frac{\alpha}{4} \lambda^2 - \beta_1 \lambda - \beta_2. 
\end{equation*}
We then have
\begin{align*}
	\mu(F) - \gamma(F) &= \frac{\alpha}{4} \lambda + \beta_1 + \frac{1}{\lambda}\pa{\beta_2 + \esps{\gamma}{\frac{m}{q} h}}, \\
	& \overset{\eqref{var_entropy}}{\le} \frac{\alpha}{4} \lambda + \beta_1 + \frac{1}{\lambda}\pa{\beta_2 + \ent_{\gamma}\pa{\frac{m}{q}}}.
\end{align*}
An optimization in $\lambda$ yields 
\begin{equation} \label{control_W1}
	\mu(F) - \gamma(F) \le \beta_1 + \sqrt{\alpha \pa{\beta_2 + \ent_{\gamma}\pa{\frac{m}{q}}}}
\end{equation}
Now using the domination condition, one has $\ent_{\gamma}\pa{\frac{m}{q}} = \int \frac{m}{q} \log \pa{\frac{m}{q}} \d \gamma \le \log(\kappa)$ and the results follows.
\end{proof}

\begin{REM}
Note that if $\gamma$ satisfies an \eqref{ISL} we have $\beta_1 = \beta_2 = 0$ and the result \eqref{control_W1} of Lemma \ref{LEMW1} reads $W_1(\mu, \gamma) \le \sqrt{\alpha \ent_{\gamma}\pa{\frac{m}{q}}}\le \sqrt{\alpha\log(\kappa)}$. For similar controls concerning the $W_2$ Wasserstein distance see Theorem 1 of Otto and Villani \cite{otto:vill:00} or Bobkov et al. \cite{bob:gen:led:01}.
\end{REM}

Using the tensorization property of the logarithmic Sobolev inequality we derive the following Corollary. Note that the term $\delta_{\kappa, \alpha}$ can be seen as a penalty term due on the one hand to the transport between $\mu$ and $\gamma$, and on the other hand to the explosion of the domination constant $\kappa^M$ between $\mu^{\otimes M}$ and $\gamma^{\otimes M}$ when $M$ tends to infinity. We emphasize that the bias $\delta_{\kappa, \alpha}$ is independent of $M$.  
Hence, the result below is especially relevant when $r$ and $\delta_{\kappa,\alpha}$ have the same order. In particular, the non-asymptotic confidence interval given by \eqref{RES_CORO_MC_UP} cannot be compared to the asymptotic confidence interval deriving from the central limit theorem whose size has order $O(M^{-1/2})$.

\begin{COROL} \label{PROPn}
Let $Y^{1}, \dots, Y^{M}$ be \emph{i.i.d.} $\R^d $-valued random variables with law $\mu$. Assume there exist $\alpha>0,\ \kappa\ge 1$ and $\gamma$ such that \eqref{DOMI} holds on $\R^d $. Then, for all Lipschitz continuous function $f:\R^d\rightarrow \R$ satisfying $\abs{\nabla f}_\infty \le 1$, we have 
\begin{equation} \label{RES_CORO_MC_UP}
	\forall r > 0, \, M \ge 1, \quad \prob{\abs{\frac{1}{M}\sum_{k=1}^M f(Y^{k}) - \esp{f(Y^{1})} } \ge r + \delta_{\kappa, \alpha}} \le 2 e^{- M \frac{r^2}{\alpha}},
\end{equation}
with $\delta_{\kappa, \alpha} = 2\sqrt{\alpha \log(\kappa)} \ge 0$. 
\end{COROL}

\begin{proof}
Let $r > 0$ and $M \ge 1$. Clearly, changing $f$ into $-f$, it suffices to prove that 
\begin{equation*}
	\prob{\frac{1}{M}\sum_{k=1}^M f(Y^{k}) - \esps{\mu}{f(Y^{1})} \ge r + \delta_{\kappa, \alpha}} \le e^{- M \frac{r^2}{\alpha}}.
\end{equation*}
By tensorization, the measure $\gamma^{\otimes M}$ satisfies an \eqref{ISL} with the same constant $\alpha$ as $\gamma$, and then the probabilities $\mu^{\otimes M}$ and $\gamma^{\otimes M}$ satisfy $(\mathcal{H}_{\kappa^M, \alpha})$ on $\R^K, K=Md $.
In this case, Lemma \ref{LEMW1} gives $\sqrt{M} \delta_{\kappa, \alpha} \ge W_1(\mu^{\otimes M}, \gamma^{\otimes M}) + \sqrt{M \alpha \log(\kappa)}$ and then 
\begin{multline*}
	\prob{\frac{1}{M}\sum_{k=1}^M f(Y^{k}) - \esps{\mu}{f(Y^{1})} \ge r + \delta_{\kappa, \alpha}} \\ \le 
	\prob{\frac{1}{\sqrt{M}}\sum_{k=1}^M f(Y^{k}) - \sqrt{M} \esps{\mu}{f(Y^{1})} \ge \sqrt{M} \pa{r + \sqrt{\alpha \log(\kappa)}} + W_1(\mu^{\otimes M}, \gamma^{\otimes M})}. 
%
\end{multline*}
Applying Proposition \ref{PROP1} with the measures $\mu^{\otimes M}$ and $\gamma^{\otimes M}$, the function $F(x_1,\dots,x_M) = \frac{1}{\sqrt{M}} \sum_{k=1}^M f(x_k)$ (which satisfies $\abs{\nabla F}_{\infty} \le 1$) and $\tilde{r} = \sqrt{M} (r+\sqrt{\alpha \log(\kappa)})$ we obtain
\begin{equation*}
	\prob{\frac{1}{M}\sum_{k=1}^M f(Y^{k}) - \esps{\mu}{f(Y^{1})} \ge r + \delta_{\kappa, \alpha}} \le \kappa^M e^{-M \frac{\pa{r + \sqrt{\alpha \log(\kappa)}}^2}{\alpha}},
\end{equation*}
and we easily conclude.
\end{proof}


\begin{REM} Note that to obtain the non-asymptotic bounds of the Monte Carlo procedure \eqref{RES_CORO_MC_UP}, we successively used the concentration properties of the reference measure $\gamma$, the control of the distance $W_1(\mu, \gamma)$ given by the variational formulation of the entropy (see Lemma \ref{LEMW1}) and the tensorization property of the functional inequality satisfied by $\gamma$. The same arguments can therefore be applied to a reference measure $\gamma$ satisfying a Poincar\'e inequality.
\end{REM}

\subsection{Gaussian concentration - Lower bound}
Concerning the previous deviation rate of Proposition \ref{PROP1}, a natural question consists in understanding whether it is sharp or not. Namely, for a given function $f$ satisfying suitable growth conditions at infinity, otherwise we cannot see the asymptotic growth, do we have a lower bound of the same order, i.e. with Gaussian decay at infinity? The next proposition gives a positive answer to that question.  

\begin{PROP}
\label{SHARP_MINO}
Let $f:\R^d \rightarrow \R_+$ be a Lipschitz continuous function satisfying $\abs{\nabla f}_\infty \le 1$ and assumption $\A{G_{\rho_0,\beta}}$ for given $\rho_0,\beta>0$.

For a $\mathcal{C}^2$ function $V$ on $\R^d$ such that $e^{-V}$ is integrable with respect to $\lambda_d$ and s.t. $\exists \bar\lambda\ge 1,\ \bar\lambda  \mathbf{I_{d \times d}} \ge \Hess (V)\ge 0$, let $\gamma(\d x)=e^{-V(x)}Z^{-1}\d x$ be the associated Gibbs probability measure. 
We assume that $\exists \kappa\ge 1 $ s.t. for $|x|\ge \rho_0 $ the measures $\mu(\d x)=m(x)\d x $ and $\gamma(\d x) $ satisfy
$$m(x)\ge \kappa^{-1}e^{-V(x)}Z^{-1}.$$
Let $\bar\Lambda:=\frac {\bar \lambda}2+\frac{\sup_{s\in S^{d-1}}|V(s\rho_0)|}{\rho_0^2}+\frac{\sup_{s\in S^{d-1}}|\nabla  V(s\rho_0)|}{\rho_0}$.

We have
\begin{equation*}
\forall r> 0,\quad \probs\mu{f(Y)-\mu(f)\ge r - (W_1(\mu, \gamma) + \delta(f,\gamma))} \ge \begin{cases}\frac{K(d,A)}{Z \bar \Lambda^{d/2} \kappa} \exp\left( - \bar \Lambda\left[\frac{r}{\beta}\vee \rho_0\right]^2\right),\ d \text{ even},\\
\frac{\arccos(\theta^{-1/2})K(d,A)}{Z \bar \Lambda^{d/2} \kappa} \exp\left( - \theta\bar \Lambda\left[\frac{r}{\beta}\vee \rho_0\right]^2\right), \ \forall \theta > 1, \ d \text{ odd},
\end{cases}
\end{equation*}
with $\delta(f,\gamma) = \gamma(f)+\beta\rho_0-\underline{f},\ \underline{f}:=\inf_{s\in S^{d-1}}f(s\rho_0)$, and $K(d,A)$ defined in \eqref{K_d_A} where $A \subset S^{d-1}$ appears in $\A{G_{\rho_0,\beta}}$. 
\end{PROP}

\begin{proof} Set $E:=\{Y \in A\times [\rho_0, \infty )\}$. Here we use the convention that for $d=1$, $A\times [\rho_0,+\infty)\subset (-\infty ,-\rho_0]\cup [\rho_0,+\infty) $. Write now
\begin{align}
\probs\mu{f(Y)-\mu(f)\ge r - (W_1(\mu, \gamma) + \delta(f,\gamma))} & \ge 
\probs\mu{f(Y)-\mu(f)\ge r - (W_1(\mu, \gamma) + \delta(f,\gamma)), E} \notag \\
& \ge  \kappa^{-1} \probs{\gamma}{f(Y)\ge r -\beta \rho_0+\underline{f}, E}:=\kappa^{-1}{\cal P}. \label{PREAL_MIN}
\end{align}
Denoting $Y_0=\rho_0\pi_{S^{d-1} }(Y) $, we have
\begin{eqnarray}
{\cal P}&\ge& \probs\gamma{f(Y_0) + \bigl(f(Y)-f(Y_0)\bigr) \ge r -\beta \rho_0+\underline{f}, E},\nonumber \\
           &\overset{\A{G_{\rho_0, \beta}} }{\ge}& \probs\gamma{\beta |Y-Y_0|\ge r -\beta \rho_0+\underline{f} -f(Y_0), E}\ge \probs\gamma{|Y|\ge \frac{r -\beta\rho_0 }{\beta}+|Y_0|, E}\nonumber\\
           &\ge  &\probs\gamma{|Y|\ge \frac{r}{\beta}\vee \rho_0,\pi_{S^{d-1}}(Y) \in A} \label{CTR_DIM}.
\end{eqnarray}       
Write
\begin{eqnarray*}      
{\cal P}           &\ge & \bint{A}^{}\sigma(\d s) \bint{\rho_0\vee \frac  r \beta}^{+\infty} \d \rho \rho^{d-1}  \exp(-V(s\rho))Z^{-1}, \nonumber
\end{eqnarray*}
where $\sigma(\d s) $ stands for the Lebesgue measure of $S^{d-1}$. Now, $\Hess(V)\le \bar\lambda \mathbf{I_{d\times d}}$ yields $\forall \rho\ge \rho_0\vee \frac{r}\beta,\ |V(s\rho)|/\rho^2\le \bar\Lambda$, $\bar \Lambda :=\frac {\bar \lambda}2+\frac{\sup_{s\in S^{d-1}}|V(s\rho_0)|}{\rho_0^2}+\frac{\sup_{s\in S^{d-1}}|\nabla  V(s\rho_0)|}{\rho_0} $ and therefore
\begin{eqnarray}
{\cal P}&\ge& |A| \bint{\rho_0\vee \frac r \beta}^{+\infty} \d\rho \rho^{d-1}  \exp(-\bar\Lambda\rho^2)Z^{-1}
\ge \frac{|A|}{Z (2\bar \Lambda)^{d/2}}  \bint{(\rho_0\vee \frac{r}{\beta})(2 \bar \Lambda)^{1/2}}^{+\infty} \d\rho \rho^{d-1} \exp(- \frac{\rho^2}2) \notag \\
& = &\frac{|A|}{Z(2\bar \Lambda)^{d/2}} Q_d\left((\rho_0\vee \frac{r}{\beta})(2\bar \Lambda)^{1/2}\right).\label{DEF_C}
\end{eqnarray}
We now  have the following explicit expression:
\begin{eqnarray*}
\forall x>0,\ Q_d(x)&:=&\exp(-\frac{x^2}2) M(d,x),\\
 M(d,x)&:=&\begin{cases} \bsum{i=0}^{\frac d2-1}x^{2i}\prod_{j=i+1}^{\frac d2-1}2j, \ d \text{ even},\\
\bsum{i=0}^{\frac{d-1}2-1}x^{2i+1}\prod_{j=i}^{\frac{d-1}2-1}(2j+1)+ \prod_{j=0}^{\frac{d-1}2-1}(2j+1)\exp(\frac{x^2}2)\bint{x}^{+\infty}\exp(-\frac{\rho^2}2)\d \rho,\ d \text{ odd},
\end{cases}
\end{eqnarray*}
with the convention that $\sum_{i=0}^{-1}=0,\ \forall k\in \N,\ \prod_{j=k}^{k-1}j =1$. 

Observe now that $\int_{x}^\infty \exp(-\rho^2/2)d\rho=(2\pi)^{1/2}\P[{\cal N}(0,1)\ge x]\ge (2\pi)^{1/2}\P[Y\in {\cal K},|Y|\ge x/\cos(\tilde \theta)]:=(2\pi)^{1/2}{\cal Q}(x)$, where $Y\sim {\cal N}(\mathbf{0_{2\times 1}},\mathbf{I_{2\times 2}}) $
is a standard bidimensional Gaussian vector and ${\cal K}:=\{z\in \R^2, \langle z,e_1\rangle \ge \cos (\tilde \theta) |z|  \},\ \tilde \theta \in (0, \frac \pi 2),\ e_1=(1,0)$. Since ${\cal Q}(x)=\frac {\tilde \theta}\pi \exp(-\frac {x^2}{2\cos^2(\tilde \theta)})$, we derive that
\begin{equation*}
Q_d(x)\ge \begin{cases}
2^{d/2-1}(d/2-1)!\exp(-\frac{x^2}{2}), \ d\text{ even},\\
\frac{\tilde \theta  2^{d/2}}{\pi^{1/2}}\prod_{j=1}^{\frac{d-1}2}(j-\frac 12)\exp(-\frac{x^2}{2\cos^2(\tilde \theta)}),\ d\text{ odd},
\end{cases}
\end{equation*}
which plugged into \eqref{DEF_C} yields:
\begin{eqnarray*}
{\cal P}&\ge& \begin{cases}
\frac{K(d,A)}{Z\bar \Lambda^{d/2}}\exp\left(- \bar \Lambda\left[\frac r\beta \vee \rho_0 \right]^2\right),\ d \text{ even},\\
\frac{\tilde \theta K(d,A)}{Z\bar \Lambda^{d/2}}\exp\left(-\frac{\bar \Lambda}{\cos^2(\tilde \theta)}\left[\frac r\beta \vee \rho_0 \right]^2\right),\ d \text{ odd}.
\end{cases}
\end{eqnarray*}

\end{proof}

\begin{COROL}
\label{CRL_MINO}
Under the assumptions of Proposition \ref{SHARP_MINO}, let $Y^1,\cdots,Y^M$  be {\rm i.i.d.} $\R^d $-valued random variables with law $\mu$. 
We have $\forall r> 0,\ \forall M\ge 1$,
\begin{eqnarray*}
 \prob{\abs{\frac{1}{M}\sum_{k=1}^M f(Y^{k}) - \esps{\mu}{f(Y^{1})} } \ge r -(W_1(\mu,\gamma)+ \delta(f ,\gamma) ) } 
\ge\\
 2\times \begin{cases}
 \exp\left(- M (\bar \Lambda + \chi) \left[\frac{r}{\beta}\vee \rho_0\right]^2\right),\ \chi = \frac{1}{\rho^2_0}\log\pa{\frac{ Z \bar \Lambda^{d/2} \kappa}{K(d,A)}}_+,\ d \text{ even},\\
 \exp\left(- M (\theta \bar \Lambda + \chi) \left[\frac{r}{\beta}\vee \rho_0\right]^2\right),\ \chi =\frac{1}{\rho_0^2}\log\pa{\frac{Z \bar \Lambda^{d/2} \kappa}{K(d,A)\arccos(\theta^{-1/2})}}_+, \ \theta \in (1,+\infty), d \text{ odd},
\end{cases}
\end{eqnarray*} 
with $K(d,A)$ defined in \eqref{K_d_A}.
\end{COROL}

\textit{Proof.} We only consider $d$ even.
 By independence of the $((Y)^k)_{k \in \leftB 1,M\rightB}$, exploiting $\bigcap_{k=1}^M \{ f(Y^k) - \esps{\mu}{f(Y^1)}\ge r - (W_1(\mu,\gamma)+ \delta(f,\gamma)  ) \}\subset\{\frac{1}{M} \sum_{k=1}^M f(Y^k) - \esps{\mu}{f(Y^1)} \ge r - (W_1(\mu,\gamma)+ \delta(f,\gamma) ) \} $, we have 
\begin{equation*}
	\prob[3]{\frac{1}{M} \sum_{k=1}^M f(Y^k) - \esps{\mu}{f(Y^1)} \ge r - (W_1(\mu,\gamma)+ \delta(f,\gamma) )} \ge \pa{\frac{K(d,A)}{  Z\bar \Lambda^{d/2} \kappa}}^M \exp\pa{-M\bar \Lambda   \left[\frac{r}{\beta}\vee \rho_0  \right]^2}.
\end{equation*}
For $\chi = \frac{1}{\rho^2_0}\log\pa{\frac{ Z \bar \Lambda^{d/2} \kappa}{K(d,A)}}_+$
, we thus obtain
\begin{equation*}
	\prob{\frac{1}{M} \sum_{k=1}^M f(Y^k) - \esps{\mu}{f(Y^1)} \ge r - (W_1(\mu,\gamma)+ \delta(f,\gamma) )} \ge  \exp\pa{- M (\bar \Lambda + \chi) \left[\frac{r}{\beta}\vee \rho_0  \right]^2},
\end{equation*}
which completes the proof.

\qed.

\subsection{Proofs of Theorem \ref{MTHM} and Corollary \ref{corol-asiat}}
\label{PROOF_MTHM}
\begin{trivlist} 
\item[-] \emph{Theorem \ref{MTHM} - Upper bound \eqref{MAJO-THM}.} 

In case $(a)$, the Gaussian probability $\gamma_{c,T}$ with density $p_{c}(T, x, .)$ defined in Theorem \ref{Aronson_Euler} satisfies a logarithmic Sobolev inequality with constant $\alpha(T) = \frac{2T}{c}$. The result then follows from Theorem \ref{Aronson_Euler} and Corollary \ref{PROPn}.

In case $(b)$, $\gamma_{c,T}(\d x') = p_c(T, x, x') \d x'
= Z^{-1} e^{-V_{T,x}(x')}\d x'$ where 
\begin{equation} \label{def-VTx}
	V_{T,x}(x') = c \pa[3]{\frac{|(x')^{1,d'}-x^{1,d'}|^2}{4 T}+3\frac{|(x')^{d'+1,d}-x^{d'+1,d}-\frac{x^{1,d'}+(x')^{1,d'}} 2T|^2}{T^3}}.
\end{equation}
The Hessian matrix of $V_{T,x}$ satisfies 
\begin{equation*}
	\forall x' \in \R^d, \quad \Hess V_{T,x}(x') = \pa{\begin{array}{cc} 
	\frac{2 c}{T} \mathbf{I_{d' \times d'}} & \frac{-3 c}{T^2} \mathbf{I_{d' \times d'}}\\
	\frac{-3 c}{T^2} \mathbf{I_{d' \times d'}} & \frac{6 c}{T^3}  \mathbf{I_{d' \times d'}}\\
	\end{array}} \ge \lambda \mathbf{I_{d \times d}},
\end{equation*}
with $\lambda = \frac{c}{T} + \frac{3 c}{T^3} \pa{1-\sqrt{1+\frac{T^2}{3} + \frac{T^4}{9}}} > 0$. By Proposition \ref{THMCRITERE}, the probability $\gamma_{c,T}$ satisfies a logarithmic Sobolev inequality with constant $\alpha(T) =\frac{2T}{c}\frac{1}{1 + \frac{3}{T^2}\pa{1-\sqrt{1+\frac{T^2}{3} + \frac{T^4}{9}}}}$. We still conclude by Theorem \ref{Aronson_Euler} and Corollary \ref{PROPn}. 

\item[-] \emph{Theorem \ref{MTHM} - Lower bound \eqref{MINO-THM}.}

With the notation $p_{c^{-1}}(t-s,x,x') = Z^{-1} e^{- V_{t-s,x}(x')}$, the Hessian of the potential $V_{T,x}$ satisfies $\forall x' \in \R^d$, $\Hess V_{T, x}(x') \le \bar \lambda  \mathbf{I_{d \times d}}$ where $\bar \lambda = \frac{c^{-1}}{ T}$ in case $(a)$ and $\bar \lambda = \frac{c^{-1}}{T} + \frac{3 c^{-1}}{T^3} \pa{1+\sqrt{1+\frac{T^2}{3} + \frac{T^4}{9}}}$ in case $(b)$. Set $\gamma_{c^{-1},T}(\d x')=p_{c^{-1}}(T,x,x')\d x'$ and $\mu_T(\d x')=p^\Delta(0,T,x,x') \d x'$.  Since $\mu_T$ and $\gamma_{c,T}$ satisfy $\eqref{DOMI}$ with $\kappa = C$ and $\alpha = \alpha(T)$ defined in \eqref{DEF_ALPHA_T}, the probability $\mu_T$ satisfies \eqref{control_laplace}, and Lemma \ref{LEMW1} yields $W_1(\mu_T, \gamma_{c,T}) \le \sqrt{\alpha(T) \log(C)}$.
Now, $\gamma_{c^{-1},T} $ and $\gamma_{c,T} $ satisfy $\eqref{DOMI}$ with $\kappa=C^2$ and $\alpha=\alpha(T) $.  We therefore get from Lemma
\ref{LEMW1}, $W_1(\gamma_{c^{-1},T}, \gamma_{c,T})\le \sqrt{2 \alpha(T) \log(C)} $. Hence, $W_1(\gamma_{c^{-1},T}, \mu_T) \le 
W_1(\mu_T,\gamma_{c,T})+W_1(\gamma_{c^{-1},T},\gamma_{c,T})
\le (1+\sqrt{2}) \sqrt{\alpha(T) \log(C)}$. 
Now, by definition of $\bar \delta_{c,C,T,f}$ we have 
	$\bar{\delta}_{c,C,T,f} \ge W_1(\gamma_{c^{-1},T}, \mu_T) + \delta(f,\gamma_{c^{-1},T})$, ($\delta(f,\gamma_{c^{-1},T}) $ introduced in Proposition \ref{SHARP_MINO}) 
and Corollary \ref{CRL_MINO} yields
\begin{equation*}
	\probs[3]{x}{\frac{1}{M} \sum_{k=1}^M f(T,(X^\Delta_T)^k) - \esps{x}{f(T,X^\Delta_T)} \ge r - \bar\delta_{c,C,T,f}} 
	\ge  \exp\pa{-\frac M{\bar \alpha(T)}  \left[\frac{r}{\beta}\vee \rho_0  \right]^2},
\end{equation*}
where observing that for our Gaussian bounds $\bar \Lambda =\frac {\bar \lambda} 2$, and 
\begin{equation*}
\frac{1}{\bar \alpha(T)}=\begin{cases}
	\frac{\bar \lambda} 2+\chi,\ \chi=\frac{1}{\rho^2_0}\log\pa{\frac{2^{-d/2} Z \bar \lambda^{d/2} C}{K(d,A)}}_+  & \text{for $d$ even}, \\
	\theta \frac{\bar \lambda} 2+\chi,\ \chi=\frac{1}{\rho_0^2}\log\pa{\frac{2^{-d/2}Z \bar \lambda^{d/2} C}{K(d,A) \arccos(\theta^{-1/2})   }}_+ & \text{for $d$ odd, $\theta> 1$},
	\end{cases}
\end{equation*}
and $K(d,A)$ defined in \eqref{K_d_A}.


Observe now that in case (a), the normalization factor $Z=Z(T,d)$ associated to $p_{c^{-1}}(T,x,.) $ writes $Z=(2\pi cT)^{d/2}$. Hence, recalling that $\bar \lambda=(cT)^{-1} $, we obtain in this case
\begin{equation*}
\chi = \begin{cases}
	\frac{1}{\rho^2_0}\log\pa{\frac{\pi^{d/2}   C  }{K(d,A)                }}_+  & \text{for $d$ even}, \\
	\frac{1}{\rho_0^2}\log\pa{\frac{\pi^{d/2} C}{K(d,A) \arccos(\theta^{-1/2})}} & \text{for $d$ odd, $\theta > 1$}.
	\end{cases}
\end{equation*}
In case (b), we have $Z=(2\pi c)^{d/2}T^d, \bar \lambda=\frac 1{cT}\left(1+\frac 3 {T^2}\left[1+\sqrt {1+\frac{T^2}3+\frac {T^4} 9}  \right] \right) $ so that 
$2^{-d/2}Z\bar \lambda^{d/2}=\left( \frac \pi T  \right)^{d/2}[T^2+3(1+\sqrt {1+\frac{T^2}3+\frac {T^4} 9} )]^{d/2} $. Eventually, since in case (b) we always have $d$ even, the correction writes
\begin{equation*}
\chi =\frac{1}{\rho^2_0}\log\pa{\frac{ \left( \frac \pi T  \right)^{d/2}[T^2+3(1+\sqrt {1+\frac{T^2}3+\frac {T^4} 9})]^{d/2}  C  }{K(d,A)                }}_+ .
\end{equation*}
This completes the proof. \qed

\item [-] \emph{Proof of Corollary \ref{corol-asiat}.}

Note that the random variable $Y_T^\Delta = \T_T^{-1} X_T^\Delta$ admits the density $p_{Y}^\Delta(T, y, y') = T^{d'} p^\Delta(0,T,\T_Ty,\T_T y')$ with respect to $\lambda_d(\d y')$. By Theorem \ref{Aronson_Euler} this density is dominated by $(Z T^{d'})^{-1} e^{-V_{T,\T_T y}(\T_T y')}$ where $V_{T,x}$ is defined in \eqref{def-VTx}. The Hessian of $y' \mapsto V_{T,\T_T y}(\T_T y')$ satisfies 
\begin{equation*}
	\forall y' \in \R^d, \quad \Hess V_{T,\T_T y}(\T_T y') = \pa{\begin{array}{cc} 
	\frac{2 c}{T}  \mathbf{I_{d' \times d'}} & \frac{-3 c}{T} \mathbf{I_{d' \times d'}}\\
	\frac{-3 c}{T} \mathbf{I_{d' \times d'}} & \frac{6 c}{T}  \mathbf{I_{d' \times d'}}\\
	\end{array}} \ge \lambda  \mathbf{I_{d \times d}},
\end{equation*}
with $\lambda = \frac{c}{T}(4 - \sqrt{13})$. We still conclude by Proposition \ref{THMCRITERE}, and Corollary \ref{PROPn}.

\end{trivlist} \qed

\mysection{Derivation of the Gaussian bounds for the discretization schemes}
\label{PARAM}

\subsection{Parametrix representation of the densities}
We first derive a parametrix representation of the densities of the schemes. The key idea is to express this density in terms of iterated convolutions of the density of a scheme with frozen coefficients, that therefore admits a Gaussian density, and a suitable kernel, that has an integrable singularity. These representations have previously been obtained in Konakov and Mammen \cite{kona:mamm:00} and Konakov \textit{et al.} \cite{kona:meno:molc:09}.  

We first need to introduce some objects and notations. Let us begin with the ``frozen'' inhomogeneous scheme. 
For fixed $x,x'\in \R^{d}$,
$0\le j<j'\le N$, we define
$ \bigl( \widetilde X_{t_i}^\Delta\bigr)_{i\in\leftB j,j'\rightB}\bigl(\equiv \bigl( \widetilde X_{t_i}^{\Delta,x'}\bigr)_{i\in\leftB j,j'\rightB}\bigr)$  by
\begin{equation}
\widetilde X_{t_j}^\Delta=x,\quad \forall i\in \leftB j,j'),\ \widetilde X_{t_{i+1}}^\Delta = \widetilde X_{t_i}^\Delta+b(t_i,x')\Delta+\sigma(t_i,x')(W_{t_{i+1}}-W_{t_i})
\label{EULER_FRO}
\end{equation} 
for case (a). Note that in the above definition 
the coefficients of the process are frozen at $x'$, 
but we omit this dependence for notational convenience.
In case (b) we define $ \bigl( \widetilde {   X}^\Delta_{t_i}\bigr)_{i\in\leftB j,j'\rightB}\bigl(= \bigl( \widetilde X_{t_i}^{\Delta,x',j'}\bigr)_{i\in\leftB j,j'\rightB} \bigr)$ by $\widetilde {  X}_{t_j}^\Delta=x$, and $\forall i\in \leftB j,j') $,
\begin{equation}
\widetilde {  X}_{t_{i+1}}^\Delta = \widetilde {  X}_{t_i}^\Delta
+ \left(\begin{array}{c}b_1(t_i, x')\Delta \\
\bint{t_i}^{t_{i+1}}(\widetilde {  X}_s^\Delta)^{1,d'} \d s \end{array} \right) +B \sigma\left(t_i,x'-\left(\begin{array}{c} {\mathbf 0_{d'\times 1}} \\ (x')^{1,d'} \end{array} \right)(t_{j'}-t_i) \right)
(W_{t_{i+1}}-W_{t_i}).
\label{EULER_MODIF_FRO}
\end{equation} 
That is, in case (b) the frozen process also depends on $j'$ through an additional term in the diffusion coefficient. This correction term is needed, in order to have good continuity properties w.r.t. the underlying metric associated to $p_c$ when performing differences of the form $a(t_j,x)-a(t_j,x'-\left(\begin{array}{c} {\mathbf 0_{d'\times 1}} \\ (x')^{1,d'} \end{array} \right)(t_{j'}-t_i) ) $, see the definition \eqref{METRIC} and Sections \ref{SEC_PREUVE_AR} and \ref{SEC_PREUVE_TEC} for details.

From now on, $p^{\Delta}(t_j,t_{j'},x,\cdot)$ and $\widetilde{p}^{\Delta,t_{j'},x'}(t_j,t_{j'},x, \cdot)$ denote the transition densities between times $t_j$ and $t_{j'}$
of the discretization schemes \eqref{EUL}, \eqref{EUL_MODIF} and the ``frozen'' schemes \eqref{EULER_FRO}, \eqref{EULER_MODIF_FRO}  respectively. 

Let us introduce a discrete ``analogue'' to the inhomogeneous infinitesimal generators of the continuous objects from which we derive the kernel of the discrete parametrix representation.
For a sufficiently smooth function $\psi: \R^d \rightarrow \R$ and fixed $x'\in \R^d$, $j'\in (0,N\rightB$, define the family of operators $(L^{\Delta}_{t_j})_{j\in\leftB 0,j')}$ and $(\widetilde L^{\Delta}_{t_j})_{j\in \leftB 0,j')}\bigl(= (\widetilde L^{\Delta,t_{j'},x'}_{t_j})_{j\in \leftB 0,j')}\bigr)$ by
\begin{equation*}
	L^\Delta_{t_j} \psi(x) = \frac{\espc[1]{\psi(X_{t_j+\Delta}^\Delta) }{X_{t_j}^\Delta = x} - \psi(x)}{\Delta}, \quad \text{and} \quad 
	\widetilde L^\Delta_{t_j} \psi(x) = \frac{ \espc[1]{\psi(\widetilde X_{t_j+\Delta}^\Delta) }{\widetilde X_{t_j}^\Delta = x} - \psi(x) }{\Delta}.
\end{equation*}

Using the notation $\widetilde p^\Delta(t_j, t_{j'},x,x') = \widetilde p^{\Delta, t_{j'}, x'}(t_j, t_{j'}, x, x')$, we now define the discrete kernel $H^{\Delta}$ by
\begin{equation} 
\label{DEF_KERNEL}
H^{\Delta}(t_j,t_{j'},x,x') = \pa{L_{t_j}^\Delta - \widetilde L_{t_j}^\Delta} \widetilde {p}^{\Delta}(t_j + \Delta,t_{j'},x,x'), \quad 0 \le j < j'\le N.
\end{equation}
Note carefully that the fixed variable $x'$ appears here twice: as the final point where we consider the density and as freezing point in the previous schemes \eqref{EULER_FRO}, \eqref{EULER_MODIF_FRO}. Note also that if $j'=j+1$ \ie~$t_{j'} = t_j+\Delta$, the transition probability $\widetilde p^{\Delta, t_{j'}, x'}(t_{j+1}, t_{j+1}, .,x')$ is the Dirac measure $\delta_{x'}$ so that 
\begin{align*}
	H^{\Delta}(t_j, t_{j+1}, x, x') &= \Delta^{-1} \pa{\espc[1]{\delta_{x'}(X_{t_{j+1}}^\Delta )}{X_{t_j}^\Delta = x} - \espc[1]{\delta_{x'}(\widetilde X_{t_{j+1}}^\Delta) }{\widetilde X_{t_j}^\Delta = x} }, \\
	& = \Delta^{-1} \pa{p^{\Delta}(t_j, t_{j+1}, x, x') - \widetilde p^{\Delta, t_{j'}, x'}(t_j, t_{j+1}, x, x')}.
\end{align*}


From the previous definition \eqref{DEF_KERNEL}, for all $0\le j<j'\le N$, 
\begin{eqnarray*}
H^\Delta(t_j,t_{j'},x,x')
=\Delta^{-1}\int_{\R^d}^{} \left[{p}^\Delta -\widetilde{p}%
^{\Delta,t_{j'},x'}\right](t_j,t_{j+1},x,u) 
\widetilde{p}^{\Delta,t_{j'},x'}(t_{j+1},t_{j'},u,x')\d u.
\end{eqnarray*}
Analogously to Lemma 3.6 in \cite{kona:mamm:00} we obtain the following result. 

\begin{PROP}[Parametrix for the density of the Euler scheme] \hspace*{.2cm}\\
\label{PROP_DEV}
Assume \A{UE}, \A{SB} are in force. Then, for $0\le t_{j}<t_{j'}\le T$, 
\begin{equation}
{p}^{\Delta}(t_j,t_{j'},x,x')=\sum_{r=0}^{j'-j}\left( \widetilde{p}^{\Delta}\otimes
_{\Delta}H^{\Delta,(r)}\right) (t_j,t_{j'},x,x'),
\label{dev_dens_EUL}
\end{equation}
where the discrete time convolution type operator $\otimes _{\Delta}$ is defined by 
\begin{eqnarray*}
(g\otimes _{\Delta}f)(t_{j},t_{j'},x,x') = \sum_{k=0}^{j'-j-1}\Delta\int_{\R^d}^{}
g(t_j,t_{j+k},x,u)f(t_{j+k},t_{j'},u ,x')\d u,
\end{eqnarray*}
where $ g\otimes_\Delta H^{\Delta,(0)} = g$ and for all $r\ge 1,\ H^{\Delta,(r)}=H^\Delta\otimes_\Delta H^{\Delta,(r-1)}$ denotes the $r$-fold discrete convolution of the kernel $H^\Delta$. W.r.t. the above definition, we use the convention that $\widetilde p^\Delta\otimes_\Delta H^{\Delta,(r)}(t_j,t_j,x,x')=0, r\ge 1 $.
\end{PROP}

\subsection{Proof of the Gaussian estimates of  Theorem \ref{Aronson_Euler}}
\label{SEC_PREUVE_AR}

The key argument for the proof is given in the following lemma whose proof is postponed to Section \ref{SEC_PREUVE_TEC}.
\begin{LEMME}
\label{THE_LEMME_FOR_BOUNDS}
 There exists $c>0,C\ge 1 $ s.t. for all $0\le j<j'\le N$, for all $r\in\leftB 0,j'-j\rightB, \forall  (x,x')\in \R^d $,
\begin{equation}
\label{CRT_IMP}
|\widetilde p^\Delta\otimes_\Delta H^{\Delta,(r)} (t_j,t_{j'},x,x')|\le C^{r+1}(t_{j'}-t_j)^{r\eta/2}\prod_{i=1}^{r+1}B\left(1+\frac{(i-1)\eta}{2},\frac \eta 2\right)    p_c(t_{j'}-t_j,x,x').
\end{equation}
In the above equation $B(m,n):=\int_{0}^1 s^{m-1}(1-s)^{n-1} \d s $ stands for the $\beta$ function. 
\end{LEMME}

The upper bound in \eqref{Aronson_E} then follows from  Proposition \ref{PROP_DEV} and the asymptotics of the $\beta$ function. It is also useful to achieve the first step of the lower bound.\\

\textit{Proof of the lower bound.} We provide in this section the global lower bound in short time. W.l.o.g. we assume that $T\le 1$. This allows to substitute the constant $C$ appearing in \eqref{CRT_IMP} by a constant $c_0\le c\exp(|b|_\infty)$ uniformly for $t_{j'}-t_j\le T $. From the upper bound, we derive the lower bound in short time, on the compact sets of the underlying metric, see \eqref{METRIC} below. This gives the diagonal decay.
To get the whole bound in short time it remains to obtain the ``off-diagonal" bound. To this end a chaining argument is needed. In case (a) it is quite standard in the Markovian framework, see Chapter VII of Bass \cite{bass:97} or Kusuoka and Stroock \cite{kusu:stro:87}. In case (b), the chaining in the appendix of \cite{dela:meno:09} can be adapted to our discrete framework. We adapt below these arguments to our non Markovian setting for the sake of completeness.

Eventually, to derive the lower bound for an arbitrary fixed $T>0$ it suffices to use the bound in short time and the semigroup property of $  p_{c^{-1}}$. Naturally, the biggest is $T$, the worse is the constant in the global lower bound.

From  Proposition \ref{PROP_DEV} we have 
\begin{eqnarray}
p^\Delta(t_j,t_{j'},x,x')&\ge & \widetilde p^\Delta(t_j,t_{j'},x,x')-\sum_{r=1}^{j'-j} |\widetilde p^\Delta\otimes_\Delta H^{\Delta,(r)} (t_j,t_{j'},x,x')|\nonumber \\
                                   &\ge & c_0^{-1}   p_{c^{-1}}(t_{j'}-t_j,x,x')- c_0(t_{j'}-t_j)^{\eta/2}   p_{c}(t_{j'}-t_j,x,x'), \label{PREAL_MINO_DIAG}
\end{eqnarray}
exploiting $\widetilde  p^\Delta(t_j,t_{j'},x,x')\ge c_0^{-1}   p_{c^{-1}}(t_{j'}-t_j,x,x')$ (cf. Lemma 3.1 of \cite{kona:meno:molc:09} in case (b)) and \eqref{CRT_IMP} (replacing $C$ by $c_0$) for the last inequality. Equation \eqref{PREAL_MINO_DIAG} provides a lower bound on compact sets provided that $T$ is small enough. Precisely, denoting 
\begin{equation}
\label{METRIC}
	d_{t_{j'}-t_j}^2(x,x') = \begin{cases}
		\frac{\abs{x-x'}^2}{t_{j'} - t_j} & \text{in case (a),} \\
		\frac{\abs{(x')^{1,d'}-x^{1,d'}}^2}{2(t_{j'} - t_j)} + 6 \frac{\abs{(x')^{d'+1,d}-x^{d'+1,d}- \frac{x^{1,d'}+(x')^{1,d'}} 2(t_{j'}-t_j)}^2}{(t_{j'}-t_j)^3} & \text{in case (b),}
	\end{cases}
\end{equation}
we have that, for a given $R_0\ge 1/2$, if $d_{t_{j'}-t_j}^2(x,x')\le 2 R_0 $ and $(t_{j'}-t_j)\le T\le \left( c_0^{-2}\exp(-c^{-1}R_0)/2\right)^{2/\eta}$,
\begin{eqnarray*}
p^\Delta(t_j,t_{j'},x,x')\ge \frac{1}{(t_{j'}-t_j)^{{\rm \bf S}} }(c_0^{-1}\exp(-c^{-1}R_0)-c_0T^{\eta/2})\ge \frac{c_0^{-1}}{2(t_{j'}-t_j)^{{\rm \bf S}}}\exp(-c^{-1}R_0)
\end{eqnarray*} 
where the parameter ${\rm \bf S} $ is the intrinsic scale of the scheme. In case (a) ${\rm \bf S}=d/2$, in case (b) ${\rm \bf S} = d$. Hence, up to a modification of $c_0^{-1}$ we have that 
\begin{equation}
\label{BD_COMP}
\exists c_0\ge 1,\ \forall 0\le j<j'\le N,\ \forall (x,x')\in (\R^d)^2,\ d_{t_{j'}-t_j}^2(x,x')\le 2 R_0, \ p^\Delta(t_j,t_{j'},x,x')\ge  \frac{c_0^{-1}}{(t_{j'}-t_j)^{{\rm \bf S}}}.
\end{equation}
In particular $\exists c>0,c_0\ge 1, \forall 0\le j<j'\le N,\ \forall (x,x')\in (\R^d)^2,\ d_{t_{j'}-t_j}^2(x,x')\le 2 R_0, \ p^\Delta(t_j,t_{j'},x,x')\ge c_0^{-1}   p_{c^{-1}}(t_{j'}-t_j,x,x')$.

\textit{Chaining in case (a)}.
Let us introduce: $\forall 0\le s<t\le T,\ (x,x',y)\in (\R^d)^3,\ p^{\Delta,y}(s,t,x,x')\d x':=\P[X_t^\Delta \in \d x' |X_s^\Delta=x, X_{\phi(s)}^\Delta=y] $. Equation \eqref{BD_COMP} provides a lower bound for the density of the scheme when $s,t$ correspond to discretization times. For the chaining the first step consists in extending this result to arbitrary times $0\le s<t\le T$. Precisely, if $d_{t-s}^2(x,x') \le R_0/12 $ we prove that 
\begin{equation}
\label{BD_COMP_ARB_TIME}
\exists c_0\ge 1,\ \forall 0\le s<t\le T,\ \forall y, \ p^{\Delta,y}(s,t,x,x')\ge c_0^{-1}(t-s)^{-d/2}.
\end{equation}
If $\phi(t)=\phi(s) $, the above density is Gaussian and \eqref{BD_COMP_ARB_TIME} holds.
If $ \phi(t)=(\phi(s)+\Delta)$, equation \eqref{BD_COMP_ARB_TIME} directly follows from a convolution argument between two Gaussian random variables. Note anyhow carefully that the ``crude" convolution argument cannot be iterated $L$ times for an arbitrary large $L$. Indeed, in that case the constants would have a geometric decay. 
Thus, for  $\phi(t)-(\phi(s)+\Delta)\ge \Delta $ we write
\begin{eqnarray}
p^{\Delta,y}(s,t,x,x')&=& \bint{(\R^d)^2}^{}p^{\Delta,y}(s,\phi(s)+\Delta,x,x_1)p^\Delta(\phi(s)+\Delta,\phi(t),x_1,x_2)p^{\Delta}(\phi(t),t,x_2,x')\d x_1 \d x_2\nonumber\\
&\ge& \bint{B_R(s,t,x,x')}^{}p^{\Delta,y}(s,\phi(s)+\Delta,x,x_1)p^\Delta(\phi(s)+\Delta,\phi(t),x_1,x_2)p^{\Delta}(\phi(t),t,x_2,x')\d x_1 \d x_2\nonumber\\
\label{Mino_int}
\end{eqnarray}
where $B_R(s,t,x,x'):= \{x_1\in \R^d :d_{\phi(s)+\Delta-s}^2(x,x_1) \le R \}\times \{x_2\in \R^d: d_{t-\phi(t)}^2(x_2,x') \le R \}$ for $R>0$ to be specified later on. 
Now, for $(x_1,x_2)\in B_R(s,t,x,x') $,
\begin{eqnarray*}
d_{\phi(t)-(\phi(s)+\Delta)}^2(x_1,x_2)=\frac{|x_1-x_2|^2}{\phi(t)-(\phi(s)+\Delta)}\le \frac{2|x_1-x|^2+4|x-x'|^2+4|x_2-x'|^2}{\phi(t)-(\phi(s)+\Delta)}\le 6R+ R_0,
\end{eqnarray*}
where we used that for $\phi(t)-(\phi(s)+\Delta)\ge \Delta ,\ \frac{1}{\phi(t)-(\phi(s)+\Delta)}\le \frac{3}{t-s}$ in the last inequality. Taking $R=R_0/6$ we obtain that $\forall (x_1,x_2)\in  B_R(s,t,x,x'),\ d_{\phi(t)-(\phi(s)+\Delta)}^2(x_1,x_2)\le 2R_0$. We therefore derive from \eqref{BD_COMP} and \eqref{Mino_int} that $\exists c_0>0 $,
\begin{eqnarray*}
p^{\Delta,y}(s,t,x,x')\ge c_0^{-1}(\phi(s)+\Delta -s)^{-d/2}(t-\phi(t))^{-d/2}(\phi(t)-(\phi(s)+\Delta))^{-d/2} \bint{(\R^d)^2}^{} \I_{(x_1,x_2)\in B_R(s,t,x,x')} \d x_1 \d x_2.
\end{eqnarray*}
Since $\phi(t)-(\phi(s)+\Delta) \le t-s$ and there exists $\tilde c>0 $ s.t. $|\{x_1\in \R^d : d_{\phi(s)+\Delta-s}^2
(x,x_1)
\le R \} | \ge \tilde c(\phi(s)+\Delta-s)^{d/2},\ |\{x_2\in \R^d: d_{t-\phi(t)}^2(x_2,x')
\le R \}|\ge \tilde c(t-\phi(t))^{d/2}  $ where $|.| $ stands for the Lebesgue measure of a given set in $\R^d$, we derive \eqref{BD_COMP_ARB_TIME}
from the above equation up to a modification of $c_0$.

It now remains to do the chaining when for $0\le j<j'\le N, \ (x,x')\in (\R^d)^2 $ we have $d_{t_{j'}-t_j}^2(x,x')\ge 2 R_0\ge 1 $.
Set $L=\lceil   K d_{t_{j'}-t_j}^2(x,x')\rceil$, for $K\ge 1$ to be specified later on and $h:=(t_{j'}-t_j)/L$. Note that $L\ge 1 $. For all $i\in \leftB 0,L\rightB$ we denote $s_i=t_j+ih, y_i=x+ \frac iL (x'-x) $ so that $s_0=t_j, s_L=t_{j'},\ y_0=x, y_L=x'$. Introduce now $\rho:=d_{t_{j'}-t_j}(x,x') (t_{j'}-t_j)^{1/2}/L = |x'-x|/L$
and for all $i\in\leftB 1,L-1\rightB,\ B_i:=\{x\in \R^d: |x-y_i| \le \rho \}$. Note that with the previous definitions $\forall i\in \leftB 0,L-1\rightB,\ |y_{i+1}-y_i | = |x'-x|/L=\rho$. Thus,
\begin{equation}
\label{PROP_BALLS}
\begin{split}
\forall x_1\in B_1, \ |x-x_1|\le 2 \rho,& \ \forall i\in \leftB 1,L-2\rightB,\ (x_i,x_{i+1})\in B_i\times B_{i+1}, \ |x_i-x_{i+1}|\le 3\rho,\\
& \forall x_{L-1}\in B_{L-1},\ | x_{L-1}-x'| \le 2 \rho.
\end{split}
\end{equation}
We can now choose $K$ large enough s.t. 
\begin{equation} \label{BORNE_DIST}
3\rho/\sqrt h=3d_{t_{j'}-t_j}(x,x')/\sqrt L\le (R_0/12)^{1/2} 
\end{equation}
so that according to \eqref{BD_COMP_ARB_TIME}, denoting $x_0=x, x_L=x'$, for all $i\in \leftB 0,L-1\rightB,\ \forall y\in \R^d, (x_i ,x_{i+1})\in B_i\times B_{i+1}, \ p^{\Delta,y}(s_i,s_{i+1},x_i,x_{i+1})\ge c_0^{-1}h^{-d/2} $ (with the slight abuse of notation $B_0=\{x\}, B_L=\{x'\}$ and $p^{\Delta,y}(0,h,x_0,$ $x_1)=p^\Delta(0,h,x,x_1) $).

We have
\begin{eqnarray}
\label{useful_mino}
p^\Delta(t_j,t_{j'},x,x') \ge \E_{t_j,x}\br{\I_{ \cap_{i=1}^{L-1} X_{s_i}^\Delta\in B_i} p^{\Delta,X_{\phi(s_{L-1})}^\Delta}(s_{L-1},t_{j'},X_{s_{L-1}}^\Delta ,x')}.
\end{eqnarray}
To proceed we have to distinguish two cases: $h\ge \Delta $ and $h<\Delta $. 
\begin{trivlist}
\item[-]
If $h\ge \Delta $, write from \eqref{useful_mino},
\begin{eqnarray*}
p^\Delta(t_j,t_{j'},x,x') \ge \E_{t_j,x}\br{\I_{ \cap_{i=1}^{L-1} X_{s_i}^\Delta\in B_i} \E[p^{\Delta,X_{\phi(s_{L-1})}^\Delta}(s_{L-1},t_{j'},X_{s_{L-1}}^\Delta ,x')|X_{s_{L-1}}^\Delta,\ X_{\phi(s_{L-1})}^\Delta]}.
\end{eqnarray*}
Since we consider the events $X_{s_{L-1}}^\Delta \in B_{L-1} $, we derive from \eqref{PROP_BALLS}, \eqref{BORNE_DIST} that $|X_{s_{L-1}}^\Delta-x'|/\sqrt h \le 2\rho/\sqrt h\le 3 d_{t_{j'}-t_j}(x,x')/\sqrt L \le (R_0/12)^{1/2} $. Hence, from \eqref{BD_COMP_ARB_TIME}
\begin{align*}
p^\Delta(t_j,t_{j'},x,x') & \ge c_0^{-1}h^{-d/2}\E_{t_j,x} \br{\I_{ \cap_{i=1}^{L-1} X_{s_i}^\Delta\in B_i}} \\
&= c_0^{-1}h^{-d/2}\E_{t_j,x} \br{\I_{ \cap_{i=1}^{L-2} X_{s_i}^\Delta\in B_i}\P[X_{s_{L-1}}^\Delta\in B_{L-1}|X_{s_{L-2}}^\Delta, X_{\phi(s_{L-2})}^\Delta] }.
\end{align*}
Now $\P[X_{s_{L-1}}^\Delta\in B_{L-1}|X_{s_{L-2}}^\Delta, X_{\phi(s_{L-2})}^\Delta]=\int_{B_{L-1}}^{} p^{\Delta,X_{\phi(s_{L-2})}^\Delta}(s_{L-2},s_{L-1},X_{s_{L-2}}^\Delta,y) \d y$, but since we restrict to $X_{s_{L-2}}^\Delta \in B_{L-2}$, according to \eqref{PROP_BALLS}, we have for all $y \in B_{L-1}, \ |X_{s_{L-2}}^\Delta-y|/\sqrt h \le 3\rho/\sqrt h\le (R_0/12)^{1/2} $ for the previous $R$ and therefore \eqref{BD_COMP_ARB_TIME} yields
\begin{eqnarray*}
p^\Delta(t_j,t_{j'},x,x') \ge (c_0^{-1}h^{-d/2})^2 |B_{L-1}|\E_{t_j,x}[\I_{ \cap_{i=1}^{L-2} X_{s_i}^\Delta\in B_i} ].
\end{eqnarray*}
Iterating the process we finally get
\begin{eqnarray*}
p^\Delta(t_j,t_{j'},x,x') \ge (c_0^{-1}h^{-d/2})^{L} \prod_{i=1}^{L-1}|B_{i}|.
\end{eqnarray*}
Observing that 
\begin{equation}
\label{F_MINO}
\exists \tilde c>0,\ \forall i\in \leftB 1,L-1\rightB,\ |B_i|\ge \tilde c \rho^d ,
\end{equation}
we obtain from the previous definition of $h$ and $L$:
\begin{eqnarray}
p^\Delta(t_j,t_{j'},x,x') &\ge&  (c_0^{-1}h^{-d/2})^{L} (\tilde c\rho^d)^{L-1}\label{BRN}\\ 
&\ge &c_0^{-1}(t_{j'}-t_j)^{-d/2}\exp((L-1)\log(c_0^{-1}\tilde c(\rho/\sqrt h)^d) )\ge c_0^{-1}(t_{j'}-t_j)^{-d/2}\exp(-cd_{t_{j'}-t_j}^2(x,x')) \nonumber
\end{eqnarray}
for a suitable $c$ up to a modification of $c_0$. 

\item[-] If $h< \Delta$. We have to introduce for all $k\in \leftB j,j'),\ I_k:=\{l\in \leftB 0,L-1\rightB,\ s_l \in [t_k,t_{k+1} [ \}$. Rewrite from \eqref{useful_mino}
\begin{eqnarray*}
 p^\Delta(t_j,t_{j'},x,x') &\ge &\E_{t_j,x}[\I_{ \cap_{k=j}^{j'-1}\cap_{i\in I_k}^{ } X_{s_i}^\Delta\in B_{i}}p^{\Delta,X_{\phi(s_{L-1})}^\Delta}(s_{L-1},t_{j'},X_{s_{L-1}}^\Delta ,x')  ].
\end{eqnarray*}
\end{trivlist}
Define for all $k\in \leftB j,j') $, $i\in \leftB 1,\sharp I_{k} \rightB, {I_{k}^i}\in I_{k}$ and  $t_{k}\le s_{I_{k}^1}< s_{I_{k}^2}<\cdots< s_{I_{k}^{\sharp I_{k}}} <t_{k+1}$. In particular, for all $i\in \leftB1, \sharp I_{k}-1\rightB,\ s_{I_{k}^{i+1}}-s_{I_{k}^{i}}=h $. Rewrite now,
\begin{eqnarray}
 p^\Delta(t_j,t_{j'},x,x')&\ge & \E_{t_j,x}[\I_{ \cap_{k=j}^{j'-2}\cap_{i\in I_k}^{ } X_{s_i}^\Delta\in B_{i}}\E[\I_{\cap_{i\in I_{j'-1}}^{} X_{s_i}^\Delta \in B_{i}}p^{\Delta,X_{t_{j'-1}}^\Delta}(s_{L-1},t_{j'},X_{s_{L-1}}^\Delta ,x')|\F_{s_{I_{j'-2}^{\sharp I_{j'-2}}}} ]  ].\nonumber\\
\label{EQ_CAS_2}
\end{eqnarray}
Introducing
\begin{eqnarray*}
P_{j'-1,j}:=\E[\I_{\cap_{i\in I_{j'-1}}^{} X_{s_i}^\Delta \in B_i}p^{\Delta,X_{t_{j'-1}}^\Delta}(s_{L-1},t_{j'},X_{s_{L-1}}^\Delta ,x')|\F_{s_{I_{j'-2}^{\sharp I_{j'-2}}}}]\\
 =\E[\I_{X_{s_{I_{j'-1}^1}}^\Delta\in B_{I_{j'-1}^1}}\bint{\prod_{i=2}^{\sharp I_{j'-1} } B_{I_{j'-1}^i}}^{}p^{\Delta,X_{t_{j'-1}}^\Delta}( s_{I_{j'-1}^1},s_{I_{j'-1}^2}, X_{s_{I_{j'-1}^1}}^\Delta, x_2)\\
\times \prod_{i=2}^{\sharp I_{j'-1}-1}p^{\Delta,X_{t_{j'-1}}^\Delta}(s_{I_{j'-1}^i},s_{I_{j'-1}^{i+1}}, x_{i},x_{i+1}) p^{\Delta,X_{t_{j'-1}}^\Delta}(s_{I_{j'-1}^{\sharp I_{j'-1}}}, t_{j'}, x_{\sharp I_{j'-1}},x') \prod_{i=2}^{\sharp I_{j'-1}}\d x_i|\F_{s_{I_{j'-2}^{\sharp I_{j'-2}}}}],
\end{eqnarray*}
we derive from \eqref{PROP_BALLS}, \eqref{BORNE_DIST} and \eqref{BD_COMP_ARB_TIME}
\begin{eqnarray*}
P_{j'-1,j}&\ge & (c_0^{-1}h^{-d/2})^{\sharp I_{j'-1}}\prod_{i=2}^{\sharp I_{j'-1}}|B_{I_{j'-1}^i}| \bint{B_{I_{j'-1}^1}}^{}p^{\Delta,X_{t_{j'-2}}^\Delta}(s_{I_{j'-2}^{\sharp I_{j'-2}}} ,s_{I_{j'-1}^1}, X_{s_{I_{j'-2}^{\sharp I_{j'-2}}}}^\Delta, x_1)\d x_1\\
&\overset{\eqref{F_MINO}}{\ge} &(c_0^{-1}h^{-d/2})^{\sharp I_{j'-1}} (\tilde c \rho^d)^{\sharp I_{j'-1}-1}\bint{B_{I_{j'-1}^1}}^{}p^{\Delta,X_{t_{j'-2}}^\Delta}( s_{I_{j'-2}^{\sharp I_{j'-2}}} ,s_{I_{j'-1}^1},X_{s_{I_{j'-2}^{\sharp I_{j'-2}}}}^\Delta, x_1)\d x_1.
\end{eqnarray*}
Plugging this estimate in \eqref{EQ_CAS_2} we obtain
\begin{eqnarray*}
 p^\Delta(t_j,t_{j'},x,x') &\ge &(c_0^{-1}h^{-d/2})^{\sharp I_{j'-1}} (\tilde c \rho^d)^{\sharp I_{j'-1}-1} \\
 &&\times \E_{t_j,x}[\I_{ \cap_{k=j}^{j'-2}\cap_{i\in I_k}^{ } X_{s_i}^\Delta\in B_{i}} \bint{B_{I_{j'-1}^1}}^{}p^{\Delta, X_{t_{j'-2}}^\Delta}(s_{I_{j'-2}^{\sharp I_{j'-2}}} ,s_{I_{j'-1}^1} , X_{s_{I_{j'-2}^{\sharp I_{j'-2}} }}^\Delta, x_1)   \d x_1  ]\\
 &\ge & (c_0^{-1}h^{-d/2})^{\sharp I_{j'-1}+1} (\tilde c \rho^d)^{\sharp I_{j'-1}}  \E_{t_j,x}[\I_{ \cap_{k=j}^{j'-2}\cap_{i\in I_k}^{ } X_{s_i}^\Delta\in B_{i}}]
\end{eqnarray*}
using once again \eqref{BORNE_DIST}, \eqref{BD_COMP_ARB_TIME} for the last inequality. Iterating this procedure we still obtain \eqref{BRN} and can conclude as in the previous case. \\

\textit{Chaining in case (b)}.
If $d_{t-s}^2(x,x') \le {\tilde c}^{-1}R_0 $, for $\tilde c $ large enough,  we derive similarly to case (a)  that
\begin{equation}
\label{BD_COMP_ARB_TIME_C}
\exists c_0>0,\ \forall 0\le s<t\le T,\ \forall y, \ p^{\Delta,y}(s,t,x,x')\ge c_0^{-1}(t-s)^{-d}.
\end{equation}
Similarly to the previous paragraph we reduce to the case $\phi(t)-(\phi(s)+\Delta)\ge \Delta $.
Then, equation \eqref{Mino_int} still holds and for the previous set $B_R$ with the current definition of $d_.^2(.,.)$. From standard computations, we derive taking a suitable $R$ that $\forall (x_1,x_2)\in  B_R(s,t,x,x'),\ d_{\phi(t)-(\phi(s)+\Delta)}^2(x_1,x_2)\le 2R_0$. 
Therefore,
\begin{equation}
\label{preal_mino_C}
p^{\Delta,y}(s,t,x,x')\ge c_0^{-1}(\phi(s)+\Delta -s)^{-d}(t-\phi(t))^{-d}(\phi(t)-(\phi(s)+\Delta))^{-d} \bint{(\R^d)^2}^{} \I_{(x_1,x_2)\in B_R(s,t,x,x')} \d x_1 \d x_2.
\end{equation}
Define now $\forall (u,y)\in (0,T]\times \R^{d},\ R>0$,  
$$\tilde B_R(u,y):=\{z\in \R^d: |z^{1,d'}-y^{1,d'}|^2/u\le R/7, |z^{d'+1,d}-y^{d'+1,d}-y^{1,d'} u|^2/u^3\le R/24 \} .$$ 
We have that $\forall z\in \tilde B_R(u,y) $:
\begin{eqnarray*}
d_u^2(y,z)&:=&\frac{|z^{1,d'}-y^{1,d'}|^2}{2 u} + 6 \frac{|z^{d'+1,d}-y^{d'+1,d}-\frac{y^{1,d'}+z^{1,d'}}2 u|^2}{u^3}\\
&\le &  12 \frac{|z^{d'+1,d}-y^{d'+1,d}-y^{1,d'} u|^2}{u^3} + 7\frac{|z^{1,d'}-y^{1,d'}|^2}{2u} \le R.
\end{eqnarray*}
Hence $\tilde B_R(\phi(s)+\Delta-s,x)\times \tilde B_R(t-\phi(t),x') \subset B_R(s,t,x,x')$ and therefore $\exists \tilde c>0, |B_R(s,t,x,x')|\ge \tilde c(t-\phi(t))^{d}(\phi(s)+\Delta-s)^{d}$ which plugged into \eqref{preal_mino_C} yields \eqref{BD_COMP_ARB_TIME_C}.\\

It now remains to do the chaining when $d_{t_{j'}-t_j}^2(x,x') \ge 2R_0$. The crucial point is to choose a ``good" path between $x$ and $x'$. In the non degenerated case it was naturally the straight line between the two points (Euclidean geodesic). In our current framework we can relate $d_{t_{j'}-t_j}^2(x,x')$ to a deterministic control problem.	Introduce:
\begin{equation}
 \tag{$\mathcal {CD}$}\label{Ctr_Det} 
 I(t_{j'}-t_j,x,x')=\inf\{ \int_{0}^{t_{j'}-t_j} |\varphi(s)|^2\d s, \phi(0)=x, \ \phi(t_{j'}-t_j)=x' \}, \ \overset{\cdot}{\phi}_t=A \phi_t + B\varphi_t,\end{equation}
with $ \ A=\left(\begin{array}{cc} \bf{0}_{d'\times d'} & \bf{0}_{d'\times d'}\\
\bf{I}_{d'\times d'} & \bf{0}_{d'\times d'} \end{array}\right), B=\left(\begin{array}{c} \bf{1}_{d'\times d'} \\ \bf{0}_{d'\times d'}\end{array} \right)
$, $\varphi\in L^2([0,t_{j'}-t_j],\R^{d'}) $. Problem \eqref{Ctr_Det} is a linear deterministic controllability problem that has a unique solution reached for 
\begin{equation}
\label{CTR_OPT_ASIAT}
\varphi_s=B^*[R(t_{j'}-t_j,s)]^*[Q_{t_{j'}-t_j}^{-1}](x'-R(t_{j'}-t_j,0)x),
\end{equation}
 where $R$ stands for the resolvent, i.e. $\forall 0\le t,t_0\le t_{j'}-t_j, \partial_t R(t,t_0)=AR(t,t_0), \ R(t_0,t_0) =\bf{I}_{d\times d} $ and $Q_{t_{j'}-t_j}=\int_{0}^{t_{j'}-t_j}   R(t_{j'}-t_j,s)BB^*R(t_{j'}-t_j,s)^*\d s $ is the Gram matrix, see e.g. Theorem 1.11 Chapter 1 in  Coron \cite{coro:07}. For \eqref{Ctr_Det} the resolvent writes $R(t,t_0)=\left(\begin{array}{cc} \bf{I}_{d'\times d'} & \bf{0}_{d'\times d'}\\ (t-t_0)\bf{I}_{d'\times d'}  &\bf{I}_{d'\times d'}\end{array}\right) $ and therefore
the Gram matrix of the  control problem corresponds to the covariance matrix of the process $X_t=x +\int_{0}^t AX_s  \d s+BW_t$ at time $t_{j'}-t_j $, that is $ Q_{t_{j'}-t_j}=\left(\begin{array}{cc}(t_{j'}-t_j)\bf{I}_{d'\times d'}& (t_{j'}-t_j)^2/2 \bf{I}_{d'\times d'} \\(t_{j'}-t_j)^2/2 \bf{I}_{d'\times d'}& (t_{j'}-t_j)^3/3 \bf{I}_{d'\times d'}\end{array}\right)$. Hence, explicit computations give:
 \begin{equation}
 \label{phi_explicite}
 \forall s\in [0,t_{j'}-t_j] ,\ \varphi_s=\frac{(x')^{1,d'}-x^{1,d'}}{(t_{j'}-t_j)^2} [6s-2(t_{j'}-t_j)]+6\frac{(x')^{d'+1,d}-x^{d'+1,d}-(x')^{1,d'}(t_{j'}-t_j)}{(t_{j'}-t_j)^3} [t_{j'}-t_j-2s],
\end{equation}
and thus,
 $\frac 12 I(t_{j'}-t_j,x,x')=d_{t_{j'}-t_j}^{2}(x,x')$ defined in \eqref{METRIC}.
 Now we have a candidate for a deterministic curve around which we can do the chaining. It is simply the deterministic curve $(\phi_s)_{s\in [0,t_{j'}-t_j]} $ solution of \eqref{Ctr_Det} for the above control $(\varphi_s)_{s\in [0,t_{j'}-t_j]} $.

To complete the proof of the chaining it remains to specify how to define the $(s_i)_{i\ge 1}, (y_i)_{i\ge 1} $ and the associated sets. Recall that $2R_0\ge 1$. We set here $L:=\lceil Kd_{t_{j'}-t_j}^2(x,x') \rceil  $ for an integer $K\ge 3$ to be specified later on. In term of the new distance, $L$ is similar in its definition to the one of the previous paragraph. 
Define $s_0=0 $, 
$s_i:=\inf\{t\in[ s_{i-1},t_{j'}-t_j]: \bint{s_{i-1}}^{t}|\varphi_s|^2 \d s= I(t_{j'}-t_j,x,x')/L\}\wedge (s_{i-1}+(t_{j'}-t_j)/L)\I_{s_{i-1}< (t_{j'}-t_j)(1-\frac{2}L)}+ (t_{j'}-t_j)\I_{s_{i-1}\ge  (t_{j'}-t_j)(1-\frac{2}L)},\ i\ge 1$. 
The previous conditions on $R_0,K$ give the well posedness of this definition.

\begin{LEMME}[Controls on the time step]
\label{LEMME_PAS}
Set for all $i\ge 0$, $\varepsilon_i:=s_{i+1}-s_i $.
There exist a constant $c_1\le 1$ and an integer  $\bar L \in [L-1,L/c_1]$, s.t. $s_{\bar L}=t_{j'}-t_j$ and 
\begin{equation}
\label{Taille_EPS}
\forall i\in \leftB 0, \bar L-2\rightB,\ c_1 \frac{t_{j'}-t_j}{L} \le \varepsilon_i\le \frac{t_{j'}-t_j}{L}, \ \frac{t_{j'}-t_j}L \le \varepsilon_{\bar L-1}\le 2\frac{t_{j'}-t_j}L.
\end{equation}
\end{LEMME}
\textit{Proof.}
We first set $\bar L = \inf \{k \geq 1 : s_k = t_{j'}-t_j\}$. The set
$\{k \geq 1 : s_k = t_{j'}-t_j\}$ is clearly non-empty.
The upper bound in \eqref{Taille_EPS} then follows from the definition of the family $(s_i)_{i\ge 1}$. Suppose now that $s_i < (t_{j'}-t_j)(1-2/L)$ for a given $0 \leq i \leq \bar L-2$. Assume also that $s_{i+1}-s_i < (t_{j'}-t_j)/L$ (otherwise
$\varepsilon_i=(t_{j'}-t_j)/L$). Then,  $\int_{s_i}^{s_{i+1}}|\varphi_s|^2 \d s =I(t_{j'}-t_j,x,x')/L$.
From \eqref{CTR_OPT_ASIAT}, \eqref{phi_explicite}, we deduce that 
$$\exists c_2>0,\ \sup_{0 \leq s \leq t_{j'}-t_j}
|\varphi_s| \leq  c_2 (t_{j'}-t_j)^{-1/2} d_{t_{j'}-t_j}(x,x').$$ 
Hence,  we obtain
\begin{equation*}
  \bint{s_i}^{s_{i+1}}|\varphi_s|^2 \d s:=\frac{I(t_{j'}-t_j,x,x')}{L}\le c_2^2 \varepsilon_i\frac{d_{t_{j'}-t_j}^2(x,x')}{(t_{j'}-t_j)}.
\end{equation*}
Recalling that $I(t_{j'}-t_j,x,x')=2d_{t_{j'}-t_j}^2(x,x') $, the lower bound in \eqref{Taille_EPS} follows  for all $i$ s.t. $s_i < T(1-2/L)$. The bound for 
$\bar L$ and the last time step are then easily derived.\qed

Define now for all $i\in \leftB 0,\bar L\rightB,\  y_i={  \phi}_{s_i}$ (in particular $y_0=x $ and $y_{\bar L}=x' $), and for all $i\in\leftB 1, \bar L-1\rightB $,
\begin{equation*}
B_i:=\{z\in\R^{d}: |Q_{K{\rho}^2}^{-1/2}(R(s_i,s_{i-1})y_{i-1}-z)|+|Q_{K{\rho}^2}^{-1/2}(z-R(s_i,s_{i+1})y_{i+1} )|\le 2R_0K^{-1/2}\},
\end{equation*}
where ${\rho}:=d_{t_{j'}-t_j}(x,x')(t_{j'}-t_j)^{1/2}/L$. Because of the transport term, we are led to consider sets that involve the forward transport  from the previous point on the optimal curve and the backward transport of the next point in the above definition. 
Equation \eqref{useful_mino} still holds with $L$ replaced by $\bar L$. Following the strategy of the previous paragraph concerning the conditioning, the end of the proof relies on the following 
\begin{LEMME}[Controls for the chaining]
\label{CTRC}
With the previous assumptions and definitions we have that for $K$ large enough:
\begin{equation}
\label{GOOD_SCALE_CPT_LB}
\begin{split}
\forall i\in\leftB 1, \bar L-2\rightB, \ \forall (x_i,x_{i+1}) \in B_i\times B_{i+1},\\
d^2_{\varepsilon_i}(x_i,x_{i+1}) \le 2R_0,\\
\forall x_1\in B_1,\ d_{s_1}^2 (x, x_{1})\le 2R_0,\\
 \forall x_{\bar L-1}\in B_{\bar L-1},   d_{\varepsilon_{\bar L-1}}^2( x_{\bar L-1} ,x')\le 2R_0.
\end{split}
\end{equation}
For the same $c_1$ as in Lemma \ref{LEMME_PAS}, 
\begin{equation}
\label{Taille_Boule}
  \forall i\in \leftB 1,\bar L-1\rightB, |B_i|\ge   c_1{\rho}^{2d},
\end{equation}
where $|B_i| $ stands for the Lebesgue measure of the set $B_i$. 
\end{LEMME}

Indeed, exploiting, \eqref{Taille_EPS}, \eqref{GOOD_SCALE_CPT_LB} (resp. \eqref{Taille_Boule}) instead of  \eqref{PROP_BALLS}, \eqref{BORNE_DIST} (resp. \eqref{F_MINO}), the proof remains unchanged. The proof of Lemma \ref{CTRC} is postponed to Section \ref{SEC_PREUVE_TEC}. 

\subsection{Proof of the technical Lemmas}
\label{SEC_PREUVE_TEC}

\subsubsection{Proof of Lemma \eqref{THE_LEMME_FOR_BOUNDS}.} 
The key estimate is the following control of the convolution kernel $H^\Delta$. There exist $c>0,C\ge 1$, s.t. for all $0\le j<j'\le N,\ x,x'\in \R^d$,
\begin{equation}
\label{CTR_KER}
|H^\Delta(t_j,t_{j'},x,x')|\le C (t_{j'}-t_{j})^{-1+\eta/2}  p_{c}(t_{j'}-t_{j} ,x,x').
\end{equation}
Indeed this bound yields that for all $0\le j<j'\le N,\ x,x'\in \R^d$
\begin{eqnarray*} 
|\widetilde p^\Delta \otimes_\Delta H^\Delta(t_j,t_{j'},x,x')|\le \Delta\bsum{k=0}^{j'-j-1}\bint{\R^d}^{}  \widetilde p^\Delta(t_j,t_{j+k},x,u) |H^\Delta(t_{j+k},t_{j'},u,x')|\d u\\
\le C^2\Delta\bsum{k=0}^{j'-j-1}(t_{j'}-t_{j+k})^{-1+\eta/2}  p_c(t_{j'}-t_j,x,x')\le C^2(t_{j'}-t_j)^{\eta/2}B\left(1,\frac\eta 2\right)  p_c(t_{j'}-t_j,x,x')
\end{eqnarray*}
using the inequality $\widetilde p^\Delta(t_{j+k}-t_j,x,u)\le C  p_c(t_j,t_{j+k},x,u) $ (cf. Lemma 3.1 of \cite{kona:meno:molc:09} in case (b)) and the semigroup property of $  p_c $ for the last but one inequality. The bound \eqref{CRT_IMP} then follows from the above control and \eqref{CTR_KER} by induction.\\

\textit{Proof of \eqref{CTR_KER}}
We consider two cases.
\begin{trivlist}
\item[-] $j'=j+1$. From \eqref{DEF_KERNEL} we have in this case, $\forall x,x'\in \R^d $,
\begin{eqnarray*}
H^{\Delta} (t_j,t_{j'},x,x')=\Delta^{-1}(p^\Delta-\widetilde p^\Delta )(t_j,t_{j'},x,x')
\end{eqnarray*}
which are Gaussian densities.  In case (a) we have $$
H^{\Delta} (t_j,t_{j'},x,x')=\Delta^{-1}\left(\frac{G\bigl ((\sqrt \Delta\sigma(t_j,x))^{-1}(x'-x-b(t_j,x)\Delta)\bigr)}{\bigl( \Delta^d \det(a(t_j,x))\bigr)^{1/2}}-\frac{G\bigl ((\sqrt \Delta\sigma(t_j,x'))^{-1}(x'-x-b(t_j,x')\Delta)\bigr)}{\bigl( \Delta^d \det(a(t_j,x'))\bigr)^{1/2}}\right), $$ 
where $\forall z\in \R^d,\ G(z)=\exp(-|z|^2/2)(2\pi)^{-d/2}$ stands for the density of the standard Gaussian vector of $\R^d$. In case (b) we get
\begin{multline*}
H^{\Delta} (t_j,t_{j'},x,x') = \Delta^{-1}(2\sqrt 3)^{d'}\times
\left(\frac{G\left ( \left (\begin{array}{c}( 
\Delta^{1/2}\sigma(t_j,x))^{-1}((x')^{1,d'}-x^{1,d'}-b_1(t_j,x)\Delta)\\
2\sqrt 3( \Delta^{3/2}\sigma(t_j,x))^{-1}((x')^{d'+1,d}-x^{d'+1,d}-\frac{x^{1,d'}+(x')^{1,d'}}2 \Delta)\end{array}\right)
\right)}{ \Delta^{d} \det(a(t_j,x))}\right.\\
-\left.\frac{G\left(\left(\begin{array}{c} 
( \Delta^{1/2}\sigma(t_j,(x')^\Delta))^{-1}((x')^{1,d'}-x^{1,d'}-b_1(t_j,x')\Delta)\\
2\sqrt 3 ( \Delta^{3/2}\sigma(t_j,(x')^\Delta))^{-1}((x')^{d'+1,d}-x^{d'+1,d}-\frac{x^{1,d'}+(x')^{1,d'}}2 \Delta)
\end{array}
\right) \right)}{ \Delta^{d} \det(a(t_j,(x')^\Delta))}\right),
\end{multline*}
where $(x')^\Delta:= x'-\left(\begin{array}{c}{\mathbf 0}_{d'\times 1}\\
(x')^{1,d'}\Delta
\end{array} \right)$ allows to have good continuity properties to equilibrate the singularities coming from the difference $|x-(x')^\Delta|\le |(x')^{1,d'}-x^{1,d'}|(1+\frac \Delta 2)+|(x')^{d'+1,d}-x^{d'+1,d} -\frac{x^{1,d'}+(x')^{1,d'}}2\Delta| $ with the terms appearing in the exponential.
In all cases, tedious but elementary computations involving the mean value theorem yield that $\exists c>0,C\ge 1 $ s.t.
\begin{equation*}
| H^{\Delta} (t_j,t_{j'},x,x')|\le C \Delta^{-1+\eta/2}  p_c(\Delta,x,x').
\end{equation*}
\item[-] $j'>j+1$. We write in case (a)
\begin{eqnarray*}
 H^{\Delta} (t_j,t_{j'},x,x')=
\Delta^{-1}\bint{\R^d}^{} G(z)\left\{\left(\widetilde p^\Delta(t_{j+1},t_{j'},x+b(t_j,x)\Delta+\Delta^{1/2}\sigma(t_j,x) z,x' )-\widetilde p^\Delta(t_{j+1},t_{j'},x,x' )\right)\right.\\
-\left. \left(\widetilde p^\Delta(t_{j+1},t_{j'},x+b(t_j,x')\Delta+\Delta^{1/2}\sigma(t_j,x')z,x' )-\widetilde p^\Delta(t_{j+1},t_{j'},x,x' )\right)\right\}\d z:=T_1^{(a)}-T_2^{(a)}.
\end{eqnarray*}
Now exploiting that $\int_{\R^d}^{}G(z)z \d z=0 $, a Taylor expansion at order 3 of $T_1^{(a)}, T_2^{(a)}$ yields
\begin{eqnarray}
\label{decomp_Kernel}
 H^{\Delta} (t_j,t_{j'},x,x')&=&\langle b(t_j,x)-b(t_j,x'),D_x \widetilde p^\Delta(t_{j+1},t_{j'},x,x')\rangle+\frac 12 {\rm Tr}\biggl(\bigl(a(t_j,x)-a(t_j,x') \bigr) D_x^2 \widetilde p^\Delta(t_{j+1},t_{j'},x,x') \biggr)\nonumber\\
  &&+R^\Delta(t_j,t_{j'},x,x'):=(H+R^\Delta)(t_j,t_{j'},x,x').
\end{eqnarray}
In the above equation $H$ is the difference of the infinitesimal generators at time $t_j$ of the processes $(X_t)_{t\ge 0} $ satisfying \eqref{SDE} and the  
Gaussian process $\tilde X_t=x+\int_{t_j}^t b(s,x') \d s+\int_{t_j}^t \sigma(s,x') \d W_s,\ t\ge t_j $, which can be seen as the continuous version of the frozen Markov chain $(\tilde X_{t_i}^{\Delta})_{i\in\leftB j,N \rightB} $ introduced in \eqref{EULER_FRO}, applied to the Gaussian density $\widetilde p^\Delta(t_{j+1},t_{j'},\cdot, x') $ at point $x$. The remainder term writes
\begin{eqnarray*}
R^\Delta(t_j,t_{j'},x,x')=\frac \Delta 2 {\rm Tr}\biggl( \bigl(bb^*(t_j,x) -bb^*(t_j,x') \bigr)D_x^2 \widetilde p^{\Delta}(t_{j+1},t_{j'},x,x')\biggr)+\\
3\Delta^{-1}\bsum{|\nu|=3}^{}\bint{\R^d}^{}\d zG(z) \bint{0}^{1}\d\delta (1-\delta)^2\left[D_x^\nu\widetilde p^\Delta(t_{j+1},t_{j'},x+\delta(b(t_j,x)\Delta+\sigma(t_j,x)\Delta^{1/2}z) ,x')\frac{(b(t_j,x)\Delta+\sigma(t_j,x)\Delta^{1/2}z)^\nu}{\nu!}\right. \\
-\left. D_x^\nu\widetilde p^\Delta(t_{j+1},t_{j'},x+\delta(b(t_j,x')\Delta+\sigma(t_j,x')\Delta^{1/2}z) ,x')\frac{(b(t_j,x')\Delta+\sigma(t_j,x')\Delta^{1/2}z)^\nu}{\nu!}\right]
\end{eqnarray*}
using the following notations for multi-indices and powers. For
$\nu =(\nu _{1},...,\nu _{d})\in \N^{d},\ x=(x_{1},...,x_{d})^{* }$ set
$| \nu | =\nu_{1}+... +\nu _{d},\ \nu!=\nu_1!...\nu_{d}!$, $(x)^{\nu
}=x_{1}^{\nu _{1}} ...\ x_{d}^{\nu_{d}},
\ D_x^{\nu }=D_{x_{1}}^{\nu
_{1}} ... D_{x_{d}}^{\nu _{d}}$. 
Recalling the standard control 
\begin{equation}
\label{der_Gauss}
\exists c>0,C\ge 1,\  \forall \nu, |\nu|\le 4, \forall  0\le j<j'\le N,\ (x,x')\in  (\R^d)^2
,\ |D_x^\nu \widetilde p^\Delta(t_j,t_{j'},x,x')|\le C(t_{j'}-t_j )^{-|\nu|/2}  p_c(t_{j'}-t_j,x,x')
\end{equation}
for the derivatives of Gaussian densities, we obtain:
\begin{eqnarray*}
|R^\Delta(t_j,t_{j'},x,x')|\le C \Delta |b|_\infty^2(t_{j'}-t_{j+1})^{-1}  p_c(t_{j'}-t_{j+1} ,x,x') +3\Delta^{-1}\biggl|\bsum{|\nu|=3}^{}\bint{\R^d}^{}\d zG(z) \bint{0}^{1}\d\delta (1-\delta)^2\times\\
\left[D_x^\nu\widetilde p^\Delta(t_{j+1},t_{j'},x+\delta(b(t_j,x)\Delta+\sigma(t_j,x)\Delta^{1/2}z) ,x')\left(\frac{(b(t_j,x)\Delta+\sigma(t_j,x)\Delta^{1/2}z)^\nu}{\nu!} -\frac{(b(t_j,x')\Delta+\sigma(t_j,x')\Delta^{1/2}z)^\nu}{\nu!}\right) \right. \\
- \left(D_x^\nu\widetilde p^\Delta(t_{j+1},t_{j'},x+\delta(b(t_j,x')\Delta+\sigma(t_j,x')\Delta^{1/2}z) ,x')-D_x^\nu\widetilde p^\Delta(t_j,t_{j'},x+\delta(b(t_j,x)\Delta+\sigma(t_j,x)\Delta^{1/2}z) ,x')\right)\times\\
\left.\frac{(b(t_j,x')\Delta+\sigma(t_j,x')\Delta^{1/2}z)^\nu}{\nu!}\right]\biggr|
\le C \left[  p_c(t_{j'}-t_{j+1} ,x,x')+\frac{\Delta^{1/2}|x-x'|^\eta}{(t_{j'}-t_{j+1})^{3/2}}\bint{\R^d}^{}\d zG(z)\bint{0}^{1} d\delta (1-\delta)^2 |z|^3\right. \times \\
   p_c(t_{j'}-t_{j+1},x+\delta(b(t_j,x)\Delta+\sigma(t_j,x)\Delta^{1/2}z ), x' )+\frac{\Delta|x-x'|^\eta}{(t_{j'}-t_{j+1})^{2}}\bint{\R^d}^{}\d zG(z)\bint{[0,1]^2}^{} d\delta d\gamma (1-\delta)^2 |z|^4  \times\\
   p_c(t_{j'}-t_{j+1},x+\gamma\delta[(b(t_j,x)-b(t_j,x'))\Delta+(\sigma(t_j,x)-\sigma(t_j,x'))\Delta^{1/2}z ], x' ).
\end{eqnarray*}
Now, using the inequality $\forall \varepsilon\in (0,1), \ |x-x'+\rho|^2\ge |x-x'|^2(1-\varepsilon)+|\rho|^2(1-\varepsilon^{-1}), \forall \rho\in \R^d  $,  taking $\rho=\delta(b(t_j,x)\Delta+\sigma(t_j,x)\Delta^{1/2}z ) $ and $\rho= \gamma\delta[(b(t_j,x)-b(t_j,x'))\Delta+(\sigma(t_j,x)-\sigma(t_j,x'))\Delta^{1/2}z ]$ respectively in the first and second integral we get 
\begin{eqnarray*}
|R^\Delta(t_j,t_{j'},x,x')|\le \frac{C}{(1-\varepsilon)^{d/2}}   p_{(1-\varepsilon)c}(t_{j'}-t_{j+1} ,x,x')\left(1+\frac{\Delta^{1/2}|x-x'|^\eta}{(t_{j'}-t_{j+1})^{3/2}}\bint{\R^d}^{}\d zG(z)\bint{0}^{1} \d\delta (1-\delta)^2 |z|^3\times \right.\\ 
\left. \exp\left (c\frac{|\sigma|_\infty^2 |z|^2\Delta}{t_{j'}-t_{j+1}}(\varepsilon^{-1}-1)\right)+\frac{\Delta|x-x'|^\eta}{(t_{j'}-t_{j+1})^{2}}\bint{\R^d}^{}\d zG(z)\bint{[0,1]^2}^{} \d\delta \d\gamma (1-\delta)^2 |z|^4 \exp\left (c\frac{|\sigma|_\infty^2 |z|^2\Delta}{t_{j'}-t_{j+1}}(\varepsilon^{-1}-1)\right) \right).
\end{eqnarray*}
Choosing $\varepsilon $ sufficiently close to one the above integrals are finite and therefore for different $c,C$ depending on $\varepsilon $ as well, we have
\begin{eqnarray}
\label{CTR_R}
|R^\Delta(t_j,t_{j'},x,x')|\le C   p_{c}(t_{j'}-t_{j+1} ,x,x')\left(1+\frac{|x-x'|^\eta}{(t_{j'}-t_{j+1})} \right)\le C (t_{j'}-t_{j+1})^{-1+\eta/2}  p_{c}(t_{j'}-t_{j+1} ,x,x').
\end{eqnarray}
Now with the definitions of \eqref{decomp_Kernel} we also have from \eqref{der_Gauss}
\begin{eqnarray*}
|H(t_j,t_{j'},x,x')|\le C (t_{j'}-t_{j+1})^{-1+\eta/2}  p_{c}(t_{j'}-t_{j+1} ,x,x').
\end{eqnarray*}
Plugging this last estimate and \eqref{CTR_R} in \eqref{decomp_Kernel} we derive
\begin{equation*}
|H^\Delta(t_j,t_{j'},x,x')|\le C (t_{j'}-t_{j+1})^{-1+\eta/2}  p_{c}(t_{j'}-t_{j+1} ,x,x')\le C (t_{j'}-t_{j})^{-1+\eta/2}  p_{c}(t_{j'}-t_{j} ,x,x').
\end{equation*}

For case (b) we have
\begin{eqnarray*}
 H^{\Delta} (t_j,t_{j'},x,x')=
\Delta^{-1}\bint{\R^d}^{} G(z)\left\{\left(\widetilde p^\Delta(t_{j+1},t_{j'},x_\Delta+B^\Delta(t_j,x)+\Sigma^\Delta(t_j,x) z,x' )-\widetilde p^\Delta(t_{j+1},t_{j'},x_\Delta,x' )\right)\right.\\
-\left. \left(\widetilde p^\Delta(t_{j+1},t_{j'},x_\Delta+B^\Delta(t_j,x')+\Sigma^\Delta\left(t_j,(x')_{\Delta,j,j'}\right)z,x' )-\widetilde p^\Delta(t_{j+1},t_{j'},x_\Delta,x' )\right)\right\}\d z:=T_1^{(b)}-T_2^{(b)},
\end{eqnarray*}
where $x_\Delta:=x+\left(\begin{array}{c}\mathbf{0}_{d'\times 1}\\ x^{1,d'}\Delta \end{array}\right) $, $(x')_{\Delta,j,j'}:=x'-\left(\begin{array}{c} {\mathbf{0}}_{d'\times 1}\\ (x')^{1,d'}  \end{array}\right) (t_{j'}-t_j) $, and we set $\forall y\in \R^d,\ \Sigma^\Delta(t_j,y):=\left(\begin{array}{cc} \Delta^{1/2}\sigma(t_j,y) &0\\ \Delta^{3/2}\sigma(t_j,y)/2 & \Delta^{3/2}\sigma(t_j,y)/(2\sqrt 3) \end{array} \right),\ B^\Delta(t_j,y):=\left(\begin{array}{c} b_1(t_j,y) \Delta\\  b_1(t_j,y) \Delta^2/2
\end{array}\right) $. 
The strategy now relies as in case (a) on Taylor expansions.
Actually the $d'$ first component in the above expression can be handled exactly in the same way. Anyhow, in order to deal with the $d'$ last components we perform the expansion of $T_{1}^{(b)}, T_{2}^{(b)} $ around the point $x_\Delta $. We obtain:
\begin{eqnarray}
\label{decomp_Kernel_C}
 H^{\Delta} (t_j,t_{j'},x,x')&=&\langle b_1(t_j,x)-b_1(t_j,x'),D_{x^{1,d'}} \widetilde p^\Delta(t_{j+1},t_{j'},x_\Delta,x')\rangle\\
&&+ \frac 12 {\rm Tr}\left\{\left(a(t_j,x)-a\left(t_j,(x')_{\Delta,j,j'} \right)\right)D_{x^{1,d'}}^2 \widetilde p^\Delta(t_{j+1},t_{j'},x_\Delta,x')  \right\}\nonumber\\
  &&+R^\Delta(t_j,t_{j'},x_\Delta,x'):=(H+R^\Delta)(t_j,t_{j'},x_\Delta,x'),\nonumber
\end{eqnarray}
where $D_{x^{1,d'}}$ denotes the differentiation w.r.t. the first $d'$ components and similarly to \eqref{decomp_Kernel}, $H$ is the difference of the generators at time $t_j$ of the processes $(X_t)_{t\ge 0} $ satisfying \eqref{SDE} and the  
Gaussian process $\tilde X_t=x+\int_{t_j}^t \left( \begin{array}{c}b_1(s,x')\\ (\tilde X_s)^{1,d'}\end{array} \right) \d s+\int_{t_j}^t B \sigma\left(s,x'-\left(\begin{array}{c} {\mathbf 0}_{\mathbf {d'\times 1}}\\ (x')^{1,d'}\end{array}\right) (t_{j'}-s)\right)
 \d W_s,\ t\in [t_j,t_{j'}] $ (continuous version of $(\tilde X_{t_i}^{\Delta})_{i\in\leftB j,j' \rightB} $ introduced in \eqref{EULER_MODIF_FRO}), applied to the Gaussian density $\widetilde p^\Delta(t_{j+1},t_{j'},\cdot, x'   ) $ at point $x_\Delta $.
The remainder term writes
\begin{eqnarray*}
R^\Delta(t_j,t_{j'},x,x')=\left\{\frac \Delta 2 {\rm Tr}\biggl( \bigl(b_1b_1^*(t_j,x) -b_1b_1^*(t_j,x') \bigr)D_{x^{1,d'}}^2 \widetilde p^{\Delta}(t_{j+1},t_{j'},x_\Delta,x')\biggr)\right.+\\
\left.   \frac \Delta 4\left({\rm Tr}\biggl( \bigl(b_1b_1^*(t_j,x) -b_1b_1^*(t_j,x') \bigr){\Delta}+\bigl(a(t_j,x)-a(t_j,(x')_{\Delta,j,j'}) \bigr)  \right)D_{x^{1,d'},x^{d'+1,d}}^2 \widetilde p^{\Delta}(t_{j+1},t_{j'},x_\Delta,x')\biggr) \right\}\\
\left\{\Delta^{-1}\bint{\R^d}^{} \d zG(z) \bint{0}^{1}\d\gamma \left\{\left\langle D_{x^{d'+1,d}} \widetilde p^\Delta(t_{j+1},t_{j'},x_\Delta+\gamma( B^\Delta(t_j,x) +\Sigma^{\Delta}(t_j,x) z),x' ),\left( B^\Delta(t_j,x) +\Sigma^{\Delta}(t_j,x) z \right)^{d'+1,d} \right\rangle\right.\right.\\
-\left.\left. \left\langle D_{x^{d'+1,d}} \widetilde p^\Delta(t_{j+1},t_j',x_\Delta+\gamma(  B^\Delta(t_j,x') +\Sigma^{\Delta}\left(t_j,(x')_{\Delta,j,j'} \right) z),x' ),\left( B^\Delta(t_j,x') +\Sigma^{\Delta}\left(t_j,(x')_{\Delta,j,j'}\right) z \right)^{d'+1,d} \right\rangle\right\}\right\}+\\
\left\{2\Delta^{-1}\bsum{|\theta|=2}^{}\bint{\R^d}^{}\d zG(z) \bint{[0,1]^2}^{}\d\gamma \d\delta (1-\delta)\gamma^2\times\right.\\
\left[ D^\theta D_{x^{1,d}}\widetilde p^\Delta(t_{j+1},t_{j'},x_\Delta+\delta \gamma( B^\Delta(t_j,x) +\Sigma^\Delta(t_j,x) z) ,x')\frac{\left( ( B^\Delta(t_j,x) +\Sigma^\Delta(t_j,x) z) \right)^\theta}{\theta!}( B^\Delta(t_j,x) +\Sigma^\Delta(t_j,x) z)^{1,d'}\right. \\
-\left. \left. D^\theta D_{x^{1,d}}^\nu\widetilde p^\Delta(t_{j+1},t_{j'},x_\Delta+\delta \gamma( B^\Delta(t_j,x') +\Sigma^\Delta(t_j,(x')_{\Delta,j,j'})z) ,x')\frac{\left( ( B^\Delta(t_j,x') +\Sigma^\Delta(t_j,(x')_{\Delta,j,j'}) z)\right)^\theta}{\theta!}\right.\right.\\
\left.\left.\phantom{\bint{}^{}\frac B \Theta}\times( B^\Delta(t_j,x') +\Sigma^\Delta(t_j,(x')_{\Delta,j,j'}) z)^{1,d'}\right]\right\}
:=(R^{1,\Delta}+R^{2,\Delta}+R^{3,\Delta})(t_j,t_{j'},x,x').
\end{eqnarray*}
Let $\mu=(\mu_1,\cdots,\mu_{d'})\in \N^{d'},\ \nu=(\nu_1,\cdots,\nu_{d'})\in \N^{d'} $ be multi-indices. Similarly to \eqref{der_Gauss} we have,
\begin{equation}
\begin{split}
\label{der_Gauss_case_C}
\exists c>0,C\ge 1,\ \forall (\mu,\nu),\ |\mu|\le 3, |\nu|\le 4, & \forall  0\le j<j'\le N,\ (x,x')\in  \R^d\times \R^d,\ \\
& |D_{x^{1,d'}}^\nu D_{x^{d'+1,d}}^\mu\widetilde p^\Delta(t_j,t_{j'},x,x')|\le C(t_{j'}-t_j )^{-(|\nu|/2+3/2|\mu|)}  p_c(t_{j'}-t_j,x,x').
\end{split}
\end{equation}
The above control yields
\begin{equation}
\label{CTR_R_c}
|R^{1,\Delta}(t_j,t_{j'},x,x')|\le  C 
  p_{c}(t_{j'}-t_{j+1} ,x_\Delta,x')\le C 
  p_{c}(t_{j'}-t_{j} ,x,x'),
\end{equation}
up to a modification of $c,C$ in the last inequality, observing that 
\begin{eqnarray*}
\frac{| (x')^{d'+1,d}-x_\Delta^{d'+1,d}- \frac{(x')^{1,d'}+x_\Delta^{1,d'}}2(t_{j'}-t_{j+1})|^2}{(t_{j'}-t_{j+1})^3}=\frac{| (x')^{d'+1,d}-x^{d'+1,d}- \frac{(x')^{1,d'}+x^{1,d'}}2(t_{j'}-t_{j+1})-x^{1,d'}\Delta|^2}{(t_{j'}-t_{j+1})^3}\\
=\frac{| (x')^{d'+1,d}-x^{d'+1,d}- \frac{(x')^{1,d'}+x^{1,d'}}2(t_{j'}-t_{j})-\frac{(x^{1,d'} -(x')^{1,d'} )}2\Delta|^2}{(t_{j'}-t_{j+1})^3}\\
\ge (1-\varepsilon)\frac{|(x')^{d'+1,d}-x^{d'+1,d}- \frac{(x')^{1,d'}+x^{1,d'}}2(t_{j'}-t_{j}) |^2}{(t_{j'}-t_j)^3} +(1-\varepsilon^{-1})\frac{|(x')^{1,d'}-x^{1,d'}|^2\Delta^2}{4(t_{j'}-t_j)^3}, \ \forall \varepsilon\in (0,1).
\end{eqnarray*}
Let us now turn to $R^{2,\Delta}(t_j,t_{j'},x,x')$. From \eqref{der_Gauss_case_C} we obtain:
\begin{eqnarray*}
|R^{2,\Delta}(t_j,t_{j'},x,x')|\le \frac{C}{(1-\varepsilon)^{d/2}}   p_{(1-\varepsilon)c}(t_{j'}-t_{j+1} ,x_\Delta,x')\left(\frac{1}{(t_{j'}-t_{j+1})^{1/2}}+\frac{\Delta^{1/2}|x-(x')_{\Delta,j,j'}|^\eta}{(t_{j'}-t_{j+1})^{3/2}}\bint{\R^d}^{}\d zG(z) |z|\times \right.\\ 
\left. \exp\left (c|\sigma|_\infty^2 |z|^2(\varepsilon^{-1}-1)\right)+\frac{\Delta^2|x-(x')_{\Delta,j,j'}|^\eta}{(t_{j'}-t_{j+1})^3}\bint{\R^d}^{}\d zG(z) |z|^2 \exp\left (c|\sigma|_\infty^2 |z|^2(\varepsilon^{-1}-1) \right) 
\right).
\end{eqnarray*}
Taking $\varepsilon$ sufficiently close to $1$ the above integrals are finite. Also 
\begin{eqnarray}
\label{CTR_Transp}
|x-(x')_{\Delta,j,j'}|^\eta\le  C(|x^{1,d'}-(x')^{1,d'}|^\eta+|x^{d'+1,d}-(x')^{d'+1,d}-(x')^{1,d}(t_{j'}-t_j)|^\eta)\nonumber \\
\le C(|x^{1,d'}-(x')^{1,d'}|^\eta+|x^{d'+1,d}-(x')^{d'+1,d}-\frac{x^{1,d}+(x')^{1,d}}2 (t_{j'}-t_j)|^\eta+|x^{1,d'}-(x')^{1,d'}|^\eta \left(\frac{t_{j'}-t_j}2\right)^\eta),
\end{eqnarray}
\end{trivlist}
and similarly to \eqref{CTR_R_c}, $ \frac{C}{(1-\varepsilon)^{d/2}}   p_{(1-\varepsilon)c}(t_{j'}-t_{j+1} ,x_\Delta,x')\le C  p_c(t_{j'}-t_{j} ,x,x')$ up to a modification of $c,C$ in the previous r.h.s. Hence, from the definition  of $  p_c $ in case (b) we deduce:
\begin{eqnarray*}
|R^{2,\Delta}(t_j,t_{j'},x,x')|\le C(t_{j'}-t_{j})^{-1+\eta/2}  p_c(t_{j'}-t_{j} ,x,x').
\end{eqnarray*}
Thus, proceeding as in case (a) using \eqref{CTR_Transp}, \eqref{der_Gauss_case_C} to handle $R^{3,\Delta}(t_j,t_{j'},x,x') $, we eventually derive that $|R^\Delta(t_j,t_{j'},x,x')|\le C(t_{j'}-t_{j})^{-1+\eta/2}  p_c(t_{j'}-t_{j} ,x,x') $.
From the definition in equation \eqref{decomp_Kernel_C} we also obtain $|H(t_j,t_{j'},x,x')|\le C(t_{j'}-t_{j})^{-1+\eta/2}  p_c(t_{j'}-t_{j} ,x,x') $ using \eqref{CTR_Transp}, \eqref{der_Gauss_case_C} and the statement \eqref{CTR_KER} follows.
\begin{REM} Note that the time dependence in the frozen dynamics \eqref{EULER_MODIF_FRO} somehow corresponds to the backward transport of the terminal condition.
It is crucial in order to allow from \eqref{CTR_Transp} the compensation of the exploding terms associated to derivatives in $x^{1,d'} $ of order greater than $2$ and derivatives in $x^{d'+1,d} $ of order greater than $1$ appearing in the kernel $H^\Delta$. A similar construction was used in \cite{kona:meno:molc:09}.
\end{REM}

\subsection{Proof of Lemma \ref{CTRC}}
   
Let us first prove \eqref{GOOD_SCALE_CPT_LB}. We begin with $(x_i,x_{i+1})\in B_i\times B_{i+1},\ i\in \leftB 1, \bar L-2\rightB $. From Section \ref{SEC_PREUVE_AR}, one can check that $ |Q_{\varepsilon_i}^{-1/2}(R(s_{i+1},s_i)x_i-x_{i+1})|^2=2d^2_{\varepsilon_i}(x_i,x_{i+1})$. Hence,
 \begin{equation*}
\begin{split}
&Q_i:=d_{\varepsilon_i} (x_i,x_{i+1}) =\frac1{\sqrt 2}|Q_{\varepsilon_i}^{-1/2}(R(s_{i+1},s_i)x_i-x_{i+1})|\le c |Q_{\varepsilon_i}^{-1/2}(x_i-R(s_i,s_{i+1})x_{i+1})| \\
&\le c\bigl\{|Q_{\varepsilon_i}^{-1/2}(x_i-R(s_i,s_{i-1})y_{i-1})|+|Q_{\varepsilon_i}^{-1/2}(R(s_i,s_{i-1}) y_{i-1}- y_i)|
+  |Q_{\varepsilon_i}^{-1/2}(y_i-R(s_i,s_{i+1})x_{i+1})| \bigr\}\\
&:=Q_i^1+Q_i^2+Q_i^3.
\end{split}
\end{equation*}
One has
\begin{equation*}
\begin{split}
Q_i^1&\le c\bsum{j=1}^{2}\varepsilon_i^{1/2-j}|(x_i-R(s_i,s_{i-1}) y_{i-1})_j|\le c \bsum{j=1}^{2}\left(\frac{\varepsilon_i}{K{\rho}^2}\right)^{1/2-j}(K^{1/2}{\rho})^{1-2j}|(x_i-R(s_i,s_{i-1} )y_{i-1})_j|,
\end{split}
\end{equation*}
denoting for all $z\in\R^d, z_1:=z^{1,d'},z_2:=z^{d'+1,d}$ with a slight abuse of notation.
Now, from \eqref{Taille_EPS}, $\varepsilon_i/({K{\rho}^2})\ge c_1 ((t_{j'}-t_j)/L)/(Kd_{t_{j'}-t_j}^2(x,x')(t_{j'}-t_j)/L^2)=c_1\frac{L}{Kd_{t_{j'}-t_j}^2(x,x')} $. Thus, recalling $L=\lceil Kd_{t_{j'}-t_j}^2(x,x')\rceil $, $\exists c>0,\ \forall j\in\leftB 1,2\rightB,\ \left(\frac{\varepsilon_i}{K{\rho}^2}\right)^{1/2-j}\le c$ and  
\begin{equation*}
Q_i^1\le c \bsum{j=1}^{2}   
(K^{1/2}{\rho})^{1-2j}|(x_i-R(s_i,s_{i-1})y_{i-1})_j|
\le  c|Q_{K{\rho}^2}^{-1/2}(x_i-R(s_i,s_{i-1}) y_{i-1})|
\le cR_0K^{-1/2},
\end{equation*}
exploiting $x_i\in B_i $ for the last identity. The term $Q_i^3 $ could be handled in a similar way so that $Q_i^1+Q_i^3\le cR_0K^{-1/2} $.
Now $Q_i^2:=\sqrt 2d_{\varepsilon_i}(y_i,y_{i+1})\le I(s_i,s_{i+1},y_i, y_{i+1})^{1/2}\le c\left( \int_{s_i}^{s_{i+1}}|\varphi_s|^2 \d s\right)^{1/2}\le c\frac{d_{t_{j'}-t_j}(x,x')}{L^{1/2}}\le \frac{c}{K^{1/2}}$. Hence, for all $i\in\leftB 1,\bar L-2\rightB $, 
$Q_i\le 2R_0 $ for $K$ large enough independent of $t_{j'}-t_j$. Eventually, for $x_1 \in B_1, x_{\bar L-1}\in B_{\bar L-1}$ the terms $Q_0:=|Q_{\varepsilon_0}^{-1/2}(R(s_1,0)x-x_{1})| $ and $Q_{\bar L-1}:= |Q_{\varepsilon_{\bar L -1}}^{-1/2}(R(t_{j'}-t_j,s_{\bar L -1})x_{\bar L -1}-x')|\le c |Q_{\varepsilon_{\bar L-1}}^{-1/2}(x_{\bar L -1}-R(s_{\bar L -1},t_{j'}-t_j)x')|  $ can be controlled as the previous $Q_i^1,\ i\in\leftB 1,\bar L-2\rightB$ from the definitions of $B_1,B_{\bar L-1} $, so that $Q_i\le 2R_0,\ i\in \{0,\bar L-1\} $ as well. This proves  \eqref{GOOD_SCALE_CPT_LB}.

It now remains to control the Lebesgue measure of the sets $(B_i)_{i\in\leftB 1,\bar L-1\rightB} $. Define for all $i\in\leftB 1, \bar L-1\rightB,\ E_i:=\{z\in \R^{d}: |Q_{K{\rho}^2}^{-1/2}( y_i-z)
|\le 2R_0(3K^{1/2})^{-1}\}$. One has $\exists \check c:=\check c(d)>0,\ |E_i|\ge \check c{\rho}^{2d}$. Let us now prove $E_i\subset B_i $. Write, for all $z\in E_i$,
\begin{eqnarray*}
 R_i&:=&|Q_{K{\rho}^2}^{-1/2}(R(s_i,s_{i-1})y_{i-1}-z)|+|Q_{K{\rho}^2}^{-1/2}(z-R(s_i,s_{i+1})y_{i+1} )|\\
&\le &|Q_{K{\rho}^2}^{-1/2}(R(s_i,s_{i-1})y_{i-1}-y_i)|+2|Q_{K\rho^2}^{-1/2}( y_i-z)
|
 +|Q_{K{\rho}^2}^{-1/2}(y_i-R(s_i,s_{i+1}) y_{i+1} )|\\
 &:=& R_i^1+R_i^2+R_i^3. 
\end{eqnarray*} 
The previous definition of $E_i$ gives $R_i^2\le \frac{4R_0}{3K^{1/2}}$. Now, arguments similar to those used to control the above $(Q_i^1,Q_i^2)_{i\in\leftB 1,M-2\rightB}$ yield
\begin{equation*}
\begin{split}
R_i^1\le c\bsum{j=1}^{2}\left(\frac{\varepsilon_i}{K{\rho}^2} \right)^{j-1/2}\varepsilon_ i^{1/2-j}|(R(s_i,s_{i-1}) y_{i-1}-  y_i)_j|\le c\frac{d_{t_{j'}-t_j}(x',x)}{L^{1/2}}\le \frac{c}{K^{1/2}}.
\end{split}
\end{equation*}
Since the term $R_i^3 $ could be handled in the same way we deduce that for $K $ large enough and $R_0 $ large enough w.r.t. the above $c$, $R_i\le 2R_0K^{-1/2}$. Hence $E_i\subset B_i $ which completes the proof. \qed

\bibliographystyle{alpha}
\bibliography{bibli}

 \end{document}

%% file: macro.tex
\def\I{{\rm{I\!\!I}}}

\newcommand \A[1]{{\bf (#1)}}

\def\F{{\cal F}}

\def\R{{\mathbb{R}} }
\def\N{{\mathbb{N}} }
\def\E{{\mathbb{E}}  }

\def\P{{\mathbb{P}}  }

\def\I{{\mathbb{I}}}
\def\det{{\rm{det}}}

\def\bint#1^#2{\displaystyle{\int_{#1}^{#2}}}
\def\bsum#1^#2{\displaystyle{\sum_{#1}^{#2}}}

\def\xdt_#1{X_#1(\Delta t)}

\def\Hess{{\rm Hess}}

\newtheorem{THM}{Theorem}[section]

\newtheorem{PROP}{Proposition}[section]

\newtheorem{COROL}{Corollary}[section]

\newtheorem{LEMME}{Lemma}[section]
\newtheorem{REM}{Remark}[section]

\newcommand{\mysection}{\setcounter{equation}{0} \section}

%% file: fcts-vincent.tex
\DeclareMathOperator{\ent}{Ent}

\renewcommand{\d}[1]{\textup{d} #1}

\newcommand{\delimleft}[2]{\ifcase #1\or
    \bigl#2\or %
    \Bigl#2\or %
    \biggl#2\or %
    \Biggl#2\or %
    \left#2\fi}
\newcommand{\delimright}[2]{\ifcase #1\or
    \bigr#2\or %
    \Bigr#2\or %
    \biggr#2\or %
    \Biggr#2\or %
    \right#2\fi}
\newcommand{\pa}[2][5]{
    \delimleft{#1}{(} #2 \delimright{#1}{)}}
\newcommand{\br}[2][5]{
    \delimleft{#1}{[} #2 \delimright{#1}{]}}
\newcommand{\ac}[2][5]{
    \delimleft{#1}{\{} #2 \delimright{#1}{\}}}

\newcommand{\psca}[2][5]{
    {\delimleft{#1}{\langle} #2%
    \delimright{#1}{\rangle}}}
\newcommand{\abs}[2][5]{
    {\delimleft{#1}{\lvert} #2%
    \delimright{#1}{\rvert}}}

\newcommand{\prob}[2][5]{
    \P \! \br[#1]{#2}}
\newcommand{\probs}[3][5]{
    \P_{#2} \! \br[#1]{#3}}
\newcommand{\probc}[3][5]{
    \P \! \br[#1]{#2%
    \delimleft{#1}{.}%
    \vphantom{#2}\vphantom{#3}%
    \delimright{#1}{\vert} #3}}
\newcommand{\esp}[2][5]{
    \E \! \br[#1]{#2}}
\newcommand{\esps}[3][5]{
    \E_{#2} \!\br[#1]{#3}}
\newcommand{\espc}[3][5]{
    \E \! \br[#1]{#2%
    \delimleft{#1}{.} \vphantom{#2}\vphantom{#3}%
    \delimright{#1}{\vert} #3}}

\renewcommand{\le}{\leqslant}
\renewcommand{\ge}{\geqslant}
\newcommand{\ds}{\displaystyle}